\documentclass{amsart}
\usepackage{stmaryrd}
\usepackage{xy}
\usepackage{mathrsfs}
\usepackage{amscd,amssymb,color}
\usepackage{eucal}

\mathchardef\varDelta="7101

\let\iso\cong
\let\sma\wedge
\newcommand{\htp}{\simeq}
\renewcommand{\to}{\mathchoice{\longrightarrow}{\rightarrow}{\rightarrow}{\rightarrow}}

\newcommand{\cO}{{\mathcal O}}

\let\catsymbfont\mathcal
\newcommand{\aA}{{\catsymbfont{A}}}
\newcommand{\aB}{{\catsymbfont{B}}}
\newcommand{\aC}{{\catsymbfont{C}}}
\newcommand{\aD}{{\catsymbfont{D}}}

\newcommand{\aF}{{\catsymbfont{F}}}
\newcommand{\aH}{{\catsymbfont{H}}}

\newcommand{\aM}{{\catsymbfont{M}}}

\newcommand{\aO}{{\catsymbfont{O}}}

\newcommand{\aR}{{\catsymbfont{R}}}
\newcommand{\aS}{{\catsymbfont{S}}}
\newcommand{\aT}{{\catsymbfont{T}}}

\newcommand{\A}{\mathcal{A}}
\newcommand{\B}{\mathcal{B}}
\newcommand{\C}{\catsymbfont{C}}

\renewcommand{\L}{\mathrm{L}}

\newcommand{\R}{\mathcal{R}}
\newcommand{\T}{\mathcal{T}}

\newcommand{\X}{\mathcal{X}}
\newcommand{\Y}{\mathcal{Y}}
\renewcommand{\O}{\mathcal{O}}

\newcommand{\bE}{\mathbb{E}}

\newcommand{\bN}{{\mathbb{N}}}

\newcommand{\bR}{{\mathbb{R}}}
\newcommand{\bS}{\mathbb{S}}

\newcommand{\bZ}{{\mathbb{Z}}}

\def\quickop#1{\expandafter\DeclareMathOperator\csname
#1\endcsname{#1}}
\quickop{id}\quickop{Id}\quickop{colim}\quickop{hocolim}\quickop{op}
\quickop{Ar}\quickop{sing}\quickop{Hom}\quickop{w}\quickop{Ho}
\quickop{ob}\quickop{diag}\quickop{Stab}\quickop{Cat}\quickop{Mot}\quickop{Mod}
\quickop{map}\quickop{Cone}\quickop{End}\quickop{Idem}\quickop{Perf}\quickop{Ind}
\quickop{Gap}\quickop{Shv}\quickop{Spaces}\quickop{cof}\quickop{ex}\quickop{perf}
\quickop{tri}\quickop{Set}\quickop{Fib}\quickop{Nat}\quickop{haut}\quickop{holim}
\quickop{Aut}\quickop{Un}\quickop{Pic}\quickop{Mon}\quickop{gp}\quickop{inv}\quickop{Gpd}\quickop{hocofib}

\newcommand{\Man}{\aM}

\newcommand{\N}{\mathrm{N}}
\newcommand{\PrL}{\mathrm{Pr^L}}
\newcommand{\PrR}{\mathrm{Pr^R}}
\newcommand{\PR}{\mathrm{Pr}}
\newcommand{\PrLR}{\mathrm{Pr^{L,R}}}
\newcommand{\Pre}{\mathrm{Pre}}

\newcommand{\stabprcat}{\Pr_\mathrm{St}^\L}
\newcommand{\Fun}{\mathrm{Fun}} % category of functors
\newcommand{\ispec}{\spectra}
\renewcommand{\top}{\mathrm{Top}}% The (ordinary) category of topological spaces
\newcommand{\slot}{-}

\newcommand{\too}{\longrightarrow}% long rightarrow

\renewcommand{\i}{\infty}
\newcommand{\Sp}{\mathrm{Sp}}% Symmetric spectra

\newcommand{\Spinc}{Spin^{c}}
\newcommand{\Alg}{\mathrm{Alg}}

\numberwithin{equation}{section}

\newtheorem{theorem}[equation]{Theorem}
\newtheorem*{theorem*}{Theorem}
\newtheorem{corollary}[equation]{Corollary}
\newtheorem{lemma}[equation]{Lemma}
\newtheorem{proposition}[equation]{Proposition}
\theoremstyle{definition}
\newtheorem{definition}[equation]{Definition}
\newtheorem{remark}[equation]{Remark}

\newtheorem{construction}[equation]{Construction}

\xyoption{arrow}
\xyoption{matrix}
\xyoption{cmtip}
\SelectTips{cm}{}
\newcommand{\GLsym}{GL_{1}}
\newcommand{\GL}[1]{\GLsym #1}

\newcommand{\Sing}[1]{\SingOp #1}
\newcommand{\SingOp}{\Pi_{\infty}}

\newcommand{\sinf}{\Sigma^{\infty}}

\newcommand{\splus}{\sinf_{+}}

\newcommand{\xra}[1]{\xrightarrow{#1}}

\newdir{ >}{{}*!/-5pt/\dir{>}}

\newcommand{\CatOf}[1]{(\text{#1})}

\newcommand{\einfty}{E_{\infty}}

\newcommand{\heq}{\simeq}

\renewcommand{\i}{\infty}

\newcommand{\PT}{\mathrm{PT}}

\newcommand{\ptspace}{*}

\DeclareMathOperator{\spaces}{\aS}
\DeclareMathOperator{\spectra}{\mathrm{Sp}}

\newcommand{\Smash}{\wedge}

\newcommand{\Z}{\mathbb{Z}}

\begin{document}

\title[Multiplicative Thom spectra and the twisted Umkehr
map]{Parametrized spectra, multiplicative Thom spectra, and the
twisted Umkehr map} 

\author[Ando]{Matthew Ando}
\address{Department of Mathematics \\
The University of Illinois at Urbana-Champaign \\
Urbana IL 61801 \\
USA} \email{mando@math.uiuc.edu}

\author[Blumberg]{Andrew J. Blumberg}
\address{Department of Mathematics, University of Texas,
Austin, TX \ 78703}
\email{blumberg@math.utexas.edu}

\author[Gepner]{David Gepner}
\address{Department of Mathematics, Purdue University, West Lafayette, IN \ 47907}
\email{dgepner@purdue.edu}

\thanks{Ando was supported in part by NSF grant DMS-1104746.
A.~J.~Blumberg was supported in part by NSF grant DMS-0906105.  The
idea for the material in section \ref{sec:twisted} arose in the course
of an AIM SQuaRE on algebraic topology and physics.  Ando and Gepner
collaborated on this paper during the workshop \textit{K-theory and
quantum fields} at the Erwin Schr\"odinger Institute.  We thank both
the AIM and the ESI for providing superb settings in which to work.}

%\date{\draftdate}
%\subjclass[2000]{}
\begin{abstract}
We introduce a general theory of parametrized objects in the setting
 of $\i$-categories.  Although spaces and spectra parametrized over
 spaces are the most familiar examples, we establish our theory in the
 generality of objects of a presentable $\i$-category parametrized
 over objects of an $\i$-topos.  We obtain a coherent functor
 formalism describing the relationship of the various adjoint functors
 associated to base-change and symmetric monoidal structures.

 Our main applications are to the study of generalized Thom spectra.
 We obtain fiberwise constructions of twisted Umkehr maps for twisted
 generalized cohomology theories using a geometric fiberwise
 construction of Atiyah duality. In order to characterize the
 algebraic structures on generalized Thom spectra and twisted
 (co)homology, we characterize the generalized Thom spectrum as a
 categorification of the well-known adjunction between units and group
 rings.
\end{abstract}

\maketitle
\setcounter{tocdepth}{1}
\tableofcontents

\section{Introduction}

In recent work with Hopkins and Rezk \cite{ABGHR, ABGHR1}, we
introduced an $\i$-categorical approach to parametrized spaces and
spectra and showed that it provides a useful context in which to study
Thom spectra and orientations.  If $X$ is a Kan complex and $\spectra$ is
the $\i$-category of spectra, then our model for the $\i$-category of
spectra parametrized by $X$ is simply the $\i$-category
$\Fun(X^{\op}, \spectra)$ of presheaves of spectra on $X$.  Conceptually
this approach exhibits the $\i$-category $\spectra$ as the ``classifying
space'' for bundles of spectra.  In the present paper we develop this
idea to give a complete theory of $\infty$-categories $\aC$
parametrized over objects of an arbitrary $\i$-topos, and we apply this
theory to develop the multiplicative theory of Thom spectra and the
theory of twisted Umkehr maps.

The perspective that we take in this paper paper is an elaboration of
the modern perspective on parametrized homotopy
theory (explored by Hu~\cite{Hu} and beautifully expounded upon by May
and Sigurdsson~\cite{MS}) that is based on an analogy between
categories of spaces parametrized over a base space and derived
categories of sheaves over a base scheme.  In the context of algebraic
geometry, associated to morphisms of the base scheme are collections of
induced derived functors which assemble into adjoint pairs satisfying
various intricate relationships.  This data is organized into what is
often referred to as Grothendieck's six-functor formalism, and is an
essential foundation of modern work in algebraic geometry,
particularly in the context of duality phenomena.  As such, we view
the development of base change functors as the basic foundational task
when setting up a theory of parametrized objects. 

One serious issue that classically arises in this context is the
coherence of the diagrams given by this structure.  For, instance the
painstaking work of Conrad~\cite{bconrad} handles some of the issues
in Hartshorne's work~\cite{hartshorne} essentially by hand.  A start
on this in the motivic context was given by Voevodsky using his
formalism of cross-functors~\cite{voevodsky}.  Voevodsky
explains that coherence can be handled either via fibered categories
(e.g., the Grothendieck construction) or using a good theory of
2-functors.  Following these outlines, a great deal of hard work in
the motivic context has developed this coherence theory~\cite{ayoub,
cisinskideglise}.  In the case of parametrized spectra, efforts in this
direction can be found in~\cite[\S13,\S17]{MS}.

One basic point of departure for this paper is the observation that
solutions to the kind of coherence problems which arise in these sorts
of situations are precisely the sorts of issues that the
$\infty$-categorical formalism is well-placed to resolve.
Specifically, $\infty$-functors are a natural generalization of
2-functors, and at the heart of Lurie's treatment of quasicategories
is a generalization of the Grothendieck construction.  Specifically,
Lurie's approach depends on a correspondence from functors from a
$\infty$-category $\aC$ into the $\infty$-category $\Cat_\infty$ of
$\infty$-categories to $\infty$-categories fibered over $\aC$.  From
this perspective, the right way to describe these functor formalisms
is in terms of sheaves valued in the $\i$-category of symmetric
monoidal presentable $\i$-categories and symmetric monoidal functors
which admit {\em both} left and right adjoints.  Whereas the
algebro-geometric context is very difficult to study, the construction
of the functor formalism and coherence is very straightforward in the
topological setting.  One of the key technical observation is that
parametrizing over the category of spaces or more generally arbitrary
$\i$-topoi can take advantage of the fact that such categories are
accessible left exact localizations of $\i$-categories which are
freely generated under colimits.

\subsection{Objects parametrized over $\i$-topoi}

Since our work in this paper is primarily topological and
differential-geometric in nature, our motivating example will be the
case of objects parametrized over the $\infty$-category of spaces, and
we will focus on this case in the introduction.  However, in the body
of the paper we state our results in terms of an arbitrary
$\infty$-topos, equipped with the cartesian symmetric monoidal
structure.  Relevant examples of $\infty$-topoi other than spaces
include $G$-spaces for topological groups $G$, or presheaves on the
orbit $\infty$-category $\Pre(\mathrm{Orb}_G)$, or sheaves of spaces
on a Grothendieck site, such as the site associated to a topological space.

Let $\spaces$ denote the $\infty$-category of spaces.  Since $\spaces$
is freely generated under colimits by its final object (the point),
for any $\infty$-category $\aM$ with small limits, the
$\infty$-category of limit-preserving functors $\spaces^{\op} \to \aM$
is equivalent (via evaluation at the point) to $\aM$ itself.  If $\aC$
is an object of $\aM$, then we will write
\[
\aC_{/(-)}\colon \spaces^{\op}\too\aM
\]
for the resulting functor.
Now, to say that $\aC_{/ (\slot)}$ preserves limits is to say that it satisfies
descent, and so we call such a functor a \emph{sheaf} on $\spaces$ with values in
$\aM$.  For example, the $\i$-category $\widehat{\Cat}_{\i}$ of (not
necessarily small) $\i$-categories is complete, and so any
$\i$-category $\aC$ uniquely determines, and is determined by, a sheaf  
\[
\aC_{/ (\slot)}\colon \spaces^{\op}\to \widehat{\Cat}_{\i}
\]
of $\i$-categories on $\spaces$. 

If $f\colon S\to T$ is a map of spaces, we write $f^{*}$ for the
induced functor $\aC_{/T} \to \aC_{/S}$.  We now restrict attention
to \emph{presentable} $\i$-categories $\aC$, essentially without loss
of generality: if $\C$ is any $\i$-category, then $\C$ embeds fully
faithfully into $\Pre(\C)$, which is presentable.  (Although note that
making this precise involves set-theoretic technicalities.)  
Then $f^*$ has a left adjoint $f_!$ and a right
adjoint $f_*$.  In certain cases, there is even a further right
adjoint $f^!$ of $f_*$, for instance when $f$ is proper in the sense
that its homotopy fibers are compact.  This includes the case of a
smooth and proper family of manifolds $S\to T$, an important example
in the study of twisted umkehr maps.

Additionally, many examples of interest involve multiplicative structure
on $\aC$.  If $\aC^\otimes$ is a presentable symmetric monoidal
$\i$-category then $\aC_{/S}^\otimes$ has a tensor bifunctor
which commutes with colimits in each variable.
For each object $X \in \aC_{/S}$, the ``left multiplication by $X$''
functor
\[
X \otimes_S (-) \colon \aC_{/S} \to \aC_{/S}
\]
admits a right adjoint
\[
F_{S}(X,-) \colon \aC_{/S} \to \aC_{/S},
\]
and we assemble all of this structure in the following omnibus theorem.

\begin{theorem}\label{thm:main}
A presentable symmetric monoidal $\i$-category $\aC^\otimes$ uniquely
determines, and is determined by, a sheaf of presentable symmetric
monoidal $\i$-categories 
\[
\aC^\otimes_{/(-)} \colon \spaces^{\op}\longrightarrow\mathrm{CAlg}(\PrL)
\]
together with left adjoints $f_!$ to the restrictions $f^*$, for
arbitrary maps of spaces $f\colon S\to T$, and right adjoints $f^!$ to
the pushforwards $f_*$, for proper maps of spaces $f\colon S\to T$,
satisfying certain coherences and relations detailed in
Section~\ref{sec:wirthmuller-contexts}. Moreover, for any space
$S$, $\aC^\otimes_{/S}$ is equivalent to the symmetric monoidal
$\i$-category $\Fun(S^{\op},\aC)^\otimes$ of $\aC^\otimes$-valued
presheaves on $S$.
\end{theorem}

There are versions of the main theorem that hold with spaces replaced
by an arbitrary $\i$-topos and $\aC^\otimes$ a presentable
$\aO^\otimes$-monoidal $\i$-category for an $\i$-operad
$\aO^{\otimes}$ equipped with a fixed map
$\bE_1^\otimes \to\aO^\otimes$.  (See Theorems~\ref{thm:exist1} and
~\ref{thm:exist2} for the precise statements.)

In Appendix~\ref{sec:comparison}, we show when restricting to the case
of parametrized spectra, Theorem~\ref{thm:main} generalizes the
homotopical structure underlying the theory of parametrized spectra
in~\cite{MS}.  Moreover, there are distinct advantages to the
$\i$-categorical context when dealing with multiplicative structures;
these were not handled in full generality in~\cite{MS} due to the
ferocious point-set technical difficulties.

\subsection{The twisted umkehr map and multiplicative Thom spectra}

One of our primary motivations for this treatment of parametrized
homotopy theory is to characterize the multiplicative properties of
the Thom spectrum functor.  We explain our foundational results in
this direction below, but we now turn to describe the most interesting
application, the construction of twisted Umkehr maps.

We begin by recalling the construction of the Thom spectrum functor in
our framework.  Let $R$ be an $\bE_n$-ring spectrum, and let
$\Mod_{R}$ be the $\i$-category of right $R$-modules.  Within $\Mod_R$
is the full subgroupoid spanned by the invertible $R$-modules,
$\Pic_R$.  Given a space $X$ and a map $f \colon X\to\Pic_{R}$,
in \cite{ABGHR, ABGHR1, ABG} we defined the Thom spectrum of $f$ to be
the colimit $Mf$ of the composite map
\[
X \xra{f} \Pic_R \xra{} \Mod_R.
\]
Regarding such a map $\alpha$ as classifying a twisted form of the
trivial $R$-line bundle over $X$, we can consider the associated
$R$-module Thom spectrum $M\alpha$ to be the $\alpha$-twisted and
$R$-stable homotopy type of $X$, and define twisted homology and
cohomology accordingly.

\begin{definition} \label{def-f-twisted-R-homology}
Let $R$ be an $\bE_n$-ring spectrum, $n>0$, and let $\alpha\colon X\to\Pic_R$ be a map. The \emph{$\alpha$-twisted
$R$-homology and $R$-cohomology groups} of $X$ are given by
\begin{align*}
    R_{\alpha}(X) & = \pi_{0}\map_R(R,
    M\alpha) \iso \pi_{0} M\alpha\\
    R^{\alpha}(X) & = \pi_{0}\map_R(M(-\alpha),R).
\end{align*}
Here $-\alpha$ denotes the inverse of $\alpha$ in the grouplike $\bE_\infty$-space $\Pic_R$ (i.e., the involution given by taking an invertible $R$-module $M$ to its $R$-linear dual $D_R M$) and $\map_R(-,-)$ denotes the mapping space in the $\infty$-category
$\Mod_R$ of right $R$-modules. 
\end{definition}

Note that this differs in two ways from the convention used
in \cite{ABGHR1, ABG}.  First, since $\Pic_R$ need not decompose as
$\bZ\times BGL_1(R)$, so it is potentially problematic to specify maps
$f\colon X\to\Pic_R$ in terms of maps $\alpha\colon X\to BGL_1(R)$ and
$n\colon X\to\bZ$; moreover, even if $\pi_0\Pic_R\cong\bZ$, the
induced map $\Pic_R\to\bZ$ does not necessarily admit a splitting as
grouplike $\bE_\infty$-spaces.  Second, in keeping with the convention
with ordinary homology and cohomology that $R_n(\ast)\cong
R^{-n}(\ast)$, as well the convention that twisted cohomology should
be the space of sections of the associated bundle of spectra, it is
necessary to dualize the twist before taking the Thom spectrum.

Nevertheless, given a invertible bundle of $R$-modules $f\colon X\to\Pic_R$, which we view (via the inclusion $\Pic_R\to\Mod_R$) as an object of the stable $\i$-category $\Fun(X,\Mod_R)$ of bundles of $R$-modules over $X$ (which is canonically equivalent to $\Fun(X^{\op},\Mod_R)$ since $X$ is an $\i$-groupoid), we may therefore form, for any integer $n$, the suspension $\Sigma^n f\in\Fun(X,\Mod_R)$.
This is of course still a bundle of invertible $R$-modules $\Sigma^n
f\colon X\to\Pic_R$, and we write
\[
R^{n+f}(X)= R^{\Sigma^n f}(X)
\]
for the twisted cohomology of $X$ with respect to the suspended twist.

We now want to construct Umkehr maps for twisted cohomology theories.
We begin by recalling how this works in the untwisted case.  For
convenience, we switch to using exponential notation for Thom spectra;
e.g., given a twist $\alpha \colon X \to \Pic_R$ the Thom spectrum
will be written $X^{\alpha}$.  Now let $X$ be a compact manifold with
tangent bundle $T$.  The Pontryagin-Thom construction gives a stable
map   
\[
\PT (X)\colon   \bS \to X^{-T} \htp DX
\]
dual to the map $X\to \ptspace$.  If $f\colon X \to B$ is a fiber
bundle of $d$-dimensional compact manifolds with tangent bundle along the fibers
$T_{f}$, then this construction generalizes to give a stable map 
\[
\PT (f) \colon    B_{+} \to X^{-T_{f}}.
\]
If $R$ is a ring spectrum, then we get a map 
\[
R^{*} (X^{-T_{f}}) \to R^{*} (B_{+})
\]
and composing with a Thom isomorphism $R^{*+d}
X \iso R^{*} (X^{-T_{f}})$, we obtain an \emph{Umkehr} map 
\[
      R^{*+d} (X) \to R^{*} (B).
\]

Recently it has become important in a number of contexts to consider
twisted generalizations of these constructions (see for
example \cite{fw:asd,MR2438341,arixv:math/0507414}).  In our context,
we explain how to provide twisted Umkehr maps for any sufficiently multiplicative generalized
cohomology theory.  The basic strategy is as follows.  Composing, a twist
$\alpha \colon B\to \Pic_{R}$ gives rise to a twist  
\[
     X \xra{\alpha f} \Pic_{R}.
\]
If $R$ is an $E_n$ ring spectrum, the category of twists is an
$E_{n-1}$ monoidal category and we show that the Thom spectrum functor
applied to the twist lands in $E_{n-1}$ ring spectra.  In particular,
we can make sense of the generalized $R$-module Thom spectrum
$X^{-T_{f} + \alpha f}$.  Provided we can construct a twisted
Pontryagin-Thom transfer map  
\[
\PT (f,\alpha)\colon B^{\alpha} \xra{} X^{-T_{f}+\alpha f},
\]
an orientation $R^{*+d} (X) \iso R^{*} (X^{-T_{f}+\alpha f})$ then
induces the twisted Umkehr map   
\[
    R^{*+d} (X) \to R^{*} (B^{\alpha})\cong R^{*-\alpha}(B).
\]

The key idea is to show that the Pontryagin-Thom map can be realized
as the pushforward of a fiberwise map
\[
\PT (f_{/B}) \colon \bS_{B} \to D (f_{/B})
\]
along the map $p \colon \to \ptspace$ to obtain the map $\PT (f)$.  In
order to use this description, we also need to be able to provide a
geometric interpretation of $D(f_{/B})$ in various cases; notably, for
smooth and proper families of compact manifolds.  This is surprisingly
difficult from a purely homotopical viewpoint, as it involves
grappling with the functoriality of the Atiyah duality map in order to
construct a parametrized version.  Our approach involves ideas related
to Hu's study of the dualizing complex in the setting of parametrized
stable homotopy theory~\cite{Hu}.
 
Given such a fiberwise Pontryagin-Thom map, we can twist by $\alpha$
via fiberwise smashing to get the map 
\[
  \PT (f_{/B})\Smash_{B} \alpha \colon \bS_{B} \Smash_{B} \alpha \to D
  (f_{/B}) \Smash_{B} \alpha
\]
of $R$-module spectra over $B$.  Applying the pushforward $p_{!}$
associated to $p \colon B \to \ptspace$ and using the multiplicative
structure of the Thom spectrum functor now yields the map $\PT
(f,\alpha)$ as well as in many cases a geometric description of the
target.

\begin{remark} 
The basic idea that in geometric circumstances the Pontryagin-Thom map
arises from a fiberwise construction goes all the way back to the
origins of the classical Umkehr map, participating as it does in the
``families'' index theorems of Atiyah and Singer~\cite{MR43:5554}.  It
is also explicit in Becker and Gottlieb's classic
paper~\cite{BeckerGottlieb}, for example.  May and Sigurdsson have a
beautiful exposition of a geometric fiberwise construction in the
setting of a ``bundle theory'' for parametrized spectra in~\cite{MS}.
\end{remark}

\subsection{Categorification of the Picard group and multiplicative
Thom spectra}

We now return to describe our foundational results on the
multiplicative structure of the Thom spectrum functor.  The monoidal
structure we have studied so far on $\aC^\otimes_{/S}$ in
Theorem \ref{thm:main} is pointwise on $S$; i.e., induced from the
diagonal map $S \to S \times S$.  For our applications for Thom
spectra, we will develop a multiplicative theory of objects
parametrized over \emph{monoidal} spaces, where the product of
parametrized objects involves the product on the base.

Our approach involves a categorification of the notions of Picard
group.  To explain where this comes from, recall that in algebra the
units functor $\GLsym$ arises from the free/forgetful adjunction
\[
\Z[\slot]\colon \CatOf{monoids} \to \Alg (\Mod_{\Z}).
\]
The restriction of this to $\Z[\slot]\colon \CatOf{groups} \to \Alg
(\Mod_{\Z})$ is then the left adjoint of the units functor $\GLsym$.  We
will proceed by categorifying this adjunction, as follows.

Fix a suitable $\i$-operad $\aO$.  If $S$ is an $\aO$-algebra,
consider the \emph{covariant} functor  
\[
\Pre\colon \spaces \to \PrL
\]
whose value at $S$ is $\Pre(S)
= \Fun(S^{\op},\spaces) \simeq\spaces_{/S}$, and which takes $f\colon
S\to T$ to the left adjoint $f_{!}$.  This functor extends to a
symmetric monoidal functor
\[
\Pre\colon \spaces^{\otimes}\to (\PrL)^{\otimes},
\]
and so induces a functor 
\[
\Pre\colon \Alg_{\aO} (\spaces) \to \Alg_{\aO} (\PrL),
\]
where here $\Alg_{\aO}(-)$ denotes the $\i$-category of
$\aO$-algebras.  Since $\spaces$ is the unit of the symmetric monoidal
structure on $\PrL$, it is apppriate to think of $\PrL$ as the
$\i$-category $\Mod_{\spaces}$ of $\spaces$-modules, and of $\Pre$ as
the free $\spaces$-module functor, analogous to the free $\Z$-module
functor.

Let $\Alg_{\aO}^{\mathrm{gp}} (\spaces)$ be the full subcategory of
$\Alg_{\aO} (\spaces)$ on the grouplike algebra objects.  We construct 
an analogous right adjoint $\Pic$ for the functor 
\[
\Pre\colon \Alg_{\aO}^{\mathrm{gp}}(\spaces) \hookrightarrow \Alg_{\aO}
(\spaces) \xra{\Pre} \Alg_{\aO} (\PrL). 
\]

\begin{definition}
Let $\aO^{\otimes}$ be an $\i$-operad equipped with a map from
$\bE_1^{\otimes}$, and let $\R^{\otimes}$ be an $\aO$-monoidal
$\i$-category.  Define $\Pic(\R)$ to be the maximal grouplike
$\i$-groupoid in the $\aO$-monoidal $\i$-category of invertible
objects of $\R^{\otimes}$.
\end{definition}

The categorified Picard group describes the right adjoint to
$\Pre$.

\begin{theorem}\label{thm:adj}
The Picard $\i$-groupoid defines a functor
\[
\Pic \colon \Alg_{\aO} (\PrL) \to \Alg_{\aO}^{\mathrm{gp}} (\spaces),
\]
that is right adjoint to the free $\aO$-monoidal $\spaces$-module functor $\Pre$.
\end{theorem}

Lurie has proved a conjecture of Mandell \cite[6.3.5.17]{HA} which
implies that for $n>1$, $\bE_n$-algebras admit $\bE_{n-1}$-monoidal
module categories.  Applying Theorem \ref{thm:adj} in the context of
$\R^{\otimes}=\Mod_{R}^{\otimes}$ for an $\bE_{n}$-ring spectrum $R$
($n>1$), in which case $\Pic (\Mod_{R}^{\otimes}) = \Pic_{R}$, now
leads to the following multiplicative characterization of the Thom
spectrum functor in terms of the categorification of $\Pic$.

\begin{theorem}\label{thm:infmain}
The functor of $\bE_{n-1}$-monoidal presentable
$\i$-categories 
\[
    \spaces_{/\Pic_{R}} \to \Mod_{R},
\]
arising from the counit of the adjunction of Theorem~\ref{thm:adj},
is the generalized Thom spectrum functor.
\end{theorem}

An immediate corollary is the following generalization of Lewis'
theorem about multiplicative structures on Thom spectra:

\begin{theorem}
Let $R$ be an $\bE_{n}$-ring spectrum, with $n>1$.  Then $\Pic_{R}$ is
an $\bE_{n-1}$-space, and if $f\colon X \to \Pic_{R}$ is $\bE_{m}$-monoidal
for some $m<n$, then the Thom spectrum $Mf$ is an $\bE_{m}$-ring spetrum.
\end{theorem}

We also derive a characterization of the multiplicative properties of
the Thom isomorphism.  Lewis showed that an $\bE_n$-orientation gives
rise to an $\bE_n$ Thom isomorphism~\cite[7.4]{lewis-may-steinberger}.
We generalize Lewis' result as follows:

\begin{corollary}\label{thm:thomiso}
Let $R$ be an $\bE_n$-ring spectrum, $n>1$, and let $f \colon
X \to\GL{R}$ be an $\bE_m$-monoidal map for some $m<n$.  Suppose that
$Mf$ admits an $\bE_m$-orientation over a spectrum $R$, i.e., an
$\bE_m$-algebra map $Mf \to R$.  Then the composite of the Thom
diagonal and the orientation 
\[
Mf \to \splus X \sma Mf \to \splus X \sma R
\]
is an equivalence of $\bE_m$-ring spectra.
\end{corollary}

The categorification of the Picard group can itself be categorified to
produce a description of the Brauer group; we give a sketch of this
theory as well as its applications in ``twisted parametrized homotopy
theory''~\cite{douglas} in Appendix~\ref{sec:brauer}.

\subsection{Parametrized homotopy theory and the tangent bundle}

Finally, we note that from another point of view, underlying the
notion of parametrized homotopy theory is Lurie's notion of the
tangent bundle $p\colon T_\aC\to\aC$ of a presentable
$\infty$-category $\aC$ \cite{HA}.  The fiber of $p\colon T_\aC\to\aC$
over the object $S$ of $\aC$ is the stablization of the slice
$\aC_{/S}$ over $S$.  In the topological context, i.e. when $\aC$ is
the $\infty$-category of spaces, then $\aC_{/S}$ is the
$\infty$-category of spaces parametrized over $S$ and
$\Stab(\aC_{/S})$ is the $\infty$-category of spectra parametrized
over $S$.  A map $f\colon S\to T$ induces restriction maps
\[
f^*\colon \Stab(\aC_{/T})\too\Stab(\aC_{/S})
\]
which admit both left and right adjoints, namely left and right Kan
extension, or induction and coinduction, written $f_!$ and $f_*$,
respectively.

\subsection{Acknowledgments}

Our debt to Peter May and Johann Sigurdsson is obvious.  We thank them
also for many useful conversations and correspondence.  We thank David
Ben-Zvi, Dan Freed, Jacob Lurie, Mike Mandell, and Thomas Nikolaus for helpful
conversations and encouragement.  This paper was improved by very
careful readings by Anssi Lahtinen and Aaron Royer.  We are grateful
to Rune Haugseng for a very useful suggestion about pushforward.
Finally, we wish to thank our collaborators Mike Hopkins and Charles
Rezk, without whom this project would not exist.

\section{Background on $\i$-categories}

In this section we give a very brief overview of our use of the
framework of $(\infty,1)$-categories.  There are now many well-studied
models for $\i$-categories, including Rezk's Segal spaces~\cite{Rezk},
the Segal categories~\cite{HirschowitzSimpson, Tamsamani} of Simpson
and Tamsamani, the ``quasicategories'' (weak Kan complexes) of
Boardman and Vogt, and the homotopy theory of simplicial categories as
studied by Dwyer-Kan and Bergner~\cite{DwyerKan, Bergner}.  Many
models are known to be equivalent (see \cite{BergnerCompare} for a
nice discussion of the situation).  Very little of the work of this
paper depends on model-specific details; given certain basic
structural properties, one could carry out most of our arguments in
any of them.  We have chosen to use quasicategories as a model for
$\i$-categories, as developed by Joyal~\cite{Joyal} and Lurie.
Particularly for our study of multiplicative structures, we have found
it convenient to rely on the extensive treatment given in Lurie's
books \cite{HTT, HA}.

\subsection{$\i$-operads and symmetric monoidal $\i$-categories}

\label{sec:ioperads}

We now quickly review the theory of $\i$-operads as we will apply it in the
body of the paper, following \cite[\S 2]{HA}.  Let $\Gamma$ denote the
category with objects the pointed sets $\{*,1,2,\ldots,n\}$ for each
natural number $n\in\mathbb{N}$ and morphisms the pointed maps of
sets.  An $\i$-operad is then specified by an $\i$-category
$\aO^{\otimes}$ and a functor
\[
p \colon \aO^{\otimes} \to \N(\Gamma)
\]
satisfying certain conditions \cite[2.1.1.10]{HA}.

\begin{remark}
This is the generalization of the notion of a multicategory (colored
operad); to obtain the generalization of an operad we restrict to
$\i$-operads equipped with an essentially surjective functor
$\Delta^{0} \to p^{-1}(\{*,1\})$.  To make sense of this, note that
$p^{-1}(\{*,1\})$ should be thought of as the ``underlying''
$\i$-category associated to $\aO^{\otimes}$, which we'd want to contain only a
single (equivalence class of) object if we're interested in studying the
$\i$-version of an ordinary operad.
\end{remark}

The identity map $\N(\Gamma) \to \N(\Gamma)$ is an $\i$-operad; this
is the analogue of the $\bE_\infty$ operad.  More generally, we can
define a topological category $\tilde{\bE}[k]$ \cite[5.1.0.2]{HA} such
that there is a natural functor $\N(\tilde{\bE}[k]) \to \N(\Gamma)$
which is an $\i$-operad.  We refer to the resulting $\i$-operads as the
$\bE_k$ operads.

\begin{remark}
This uses a general correspondence result
which associates to a simplicial multicategory an operadic nerve which
is an $\i$-operad provided that each morphism
simplicial set of the multicategory is a Kan complex
\cite[2.1.1.27]{HA}.  
\end{remark}

A {\em symmetric monoidal} $\i$-category is then an $\i$-category
$\aC^{\otimes}$ equipped with a coCartesian fibration of $\i$-operads
\cite[2.1.2.18]{HA}
\[
p \colon \aC^{\otimes} \to \N(\Gamma).
\]
The ``underlying'' $\i$-category is obtained as the fiber $\aC =
p^{-1}(\{*,1\})$.  In abuse of terminology, we will say that
an $\i$-category $\aC$ is a symmetric monoidal $\i$-category if it is
equivalent to $p^{-1}(\{*,1\})$ for some symmetric monoidal
$\i$-category $\aC^{\otimes}$.  More generally, if $\aO^{\otimes}$ is
an $\infty$-operad and $\aC^{\otimes} \to \aO^{\otimes}$ is a
coCartesian fibration of $\i$-operads such that the composite
\[
\aC^\otimes\to\aO^\otimes\to\N(\Gamma)
\]
exhibits $\aC^\otimes$ as an $\i$-operad \cite[2.1.2.13]{HA}, then
$\aC$ is an $\aO$-monoidal $\i$-category.

Given a symmetric monoidal model category $\aC$, we can associate a
symmetric monoidal $\i$-category
$\N(\aC^\mathrm{c})[W^{-1}]^{\otimes}$ with underlying $\i$-category
$\N(\aC^\mathrm{c})[W^{-1}]$ \cite[4.1.3.6]{HA}.  (See
Appendix~\ref{sec:comparison} for further discussion of the passage
from model categories to $\i$-categories.)

For symmetric monoidal $\i$-categories $\aC^{\otimes}$ and
$\aD^{\otimes}$, we have two associated categories of functors between
them: 
\begin{enumerate} 
\item The $\i$-category of $\i$-operad maps $\Alg_{\aC}(\aD)$, which
  should be thought of as the analogue of lax symmetric monoidal
  functors \cite[2.1.2.7]{HA}, 
\item and the $\i$-category $\Fun^{\otimes}(\aC,\aD)$ of symmetric
  monoidal functors, which should be regarded as strong symmetric
  monoidal functors \cite[2.1.3.7]{HA}.
\end{enumerate}

For a fibration $q \colon \aC^{\otimes} \to \aO^{\otimes}$ of
$\i$-operads and a map of $\i$-operads
$\alpha \colon \aO'^{\otimes} \to
\aO^{\otimes}$, we define an $\aO'$-algebra object of $\aC$ over $\aO$
to be a map of $\i$-operads
$A \colon \aO'^{\otimes} \to \aC^{\otimes}$ over $\aO$ such that
$q \circ A$ is $\alpha$ \cite[2.1.3.1]{HA}.  The $\i$-category of
$\aO'$-algebra objects in $\aC$ over $\aO$, denoted $\Alg_{\aO'
/ \aO}(\aC)$, is the full subcategory of the functor category
$\Fun_{\aO^{\otimes}}(\aO'^{\otimes}, \aC^{\otimes})$ spanned by the
maps of $\i$-operads.  When $\aO$ is the commutative $\i$-operad,
$\Alg_{\aO' / \aO}(\aC) \htp \Alg_{\aO'}(\aC)$.  When $\alpha$ is the
identity map we denote this by $\Alg_{/ \aO}(\aC)$.  When in addition
$\aO$ is the commutative $\i$-operad, we will write
$\textrm{CAlg}(\aC)$ to denote the category $\Alg_{/ \aO}(\aC)$ ---
these are the commutative algebra objects in $\aC$.

A particularly interesting class of symmetric monoidal structures on
$\i$-categories come from {\em cartesian monoidal structures}.  Any
$\i$-category with finite products admits a unique cartesian symmetric
monoidal structure; the monoidal product is given by the categorical
product \cite[\S 2.4.1]{HA}.  When $\aC$ is a cartesian symmetric
monoidal $\i$-category and $\aO$ is an $\i$-operad, we often write
$\Mon_{\aO}(\aC)$ in place of $\Alg_{/ \aO}(\aC)$ \cite[\S 2.4.2]{HA}.

Finally, we will primarily be interested in algebras over $\i$-operads
which are {\em unital} and {\em coherent}.  A unital
$\i$-operad \cite[2.3.1.1]{HA} has an essentially unique ``nullary''
operation.  For instance, the $\bE_n$ operads are unital for any $n$.
As one might expect, algebras over unital operads are equipped with
well-behaved unit maps.  Coherent $\i$-operads $\aO^\otimes$ \cite[\S
3.3.1]{HA} satisfy conditions so that the categories of modules over
$\aO$-algebras have reasonable multiplicative properties.

\subsection{The $\i$-category of presentable $\i$-categories}

In this section, we quickly review the $\i$-category of presentable
$\i$-categories and introduce some notation.  An $\i$-category $\aC$
is presentable if there exists a regular cardinal $\kappa$ and a small
$\i$-category $\aD$ with $\kappa$-small colimits such that
$\aC \htp \Ind_{\kappa}(\aD)$; i.e., $\aC$ is the free completion of
$\aD$ under $\kappa$-filtered colimits~\cite[5.5.1.1]{HTT}.  The theory
of presentable $\i$-categories is an analogue of the theory of
combinatorial model categories.

We will be working with functors to the $\i$-category of presentable
$\i$-categories.  Recall from \cite[\S 6.3.1]{HA} that the
$\i$-category $\PR^\L$ of presentable $\i$-categories and left adjoint
functors is complete and cocomplete, with limits created in $\Cat_\i$.
The $\i$-category $\PR^{\L}$ is closed, as the $\i$-category of
functors $\Fun^{\L}(\aC,\aD)$ is itself a presentable
$\i$-category~\cite[5.5.3.8]{HTT}.  Furthermore, $\PR^{\L}$ is
tensored over spaces, with a convenient description of the tensor: 
Given a presentable $\i$-category $\aC$ and a space $S$, the
presentable $\i$-category $S\otimes\aC$ is naturally equivalent
$\Fun(S^{\op},\aC)$.

In addition, $\PR^\L$ can be given the structure of a symmetric
monoidal $\i$-category, with unit $\spaces$, the $\i$-category of spaces.
The fact that $\spaces$ is the unit implies that it is canonically a
commutative algebra object and that the forgetful functor 
\[
\Mod_{\spaces}(\PR^\L)\to\PrL
\]
is an equivalence, where here $\Mod_{\spaces}
= \Mod_{\spaces}^\mathrm{Comm}(\Pr^\L)$ is the $\i$-category of
modules over $\spaces$ in the symmetric monoidal structure on
$Pr^\L$.

The $\i$-category $\PR^\L$ has a subcategory the $\i$-category
$\stabprcat$ of stable presentable $\i$-categories and
colimit-preserving functors.  Recall that the $\i$-category
$\stabprcat$ of stable presentable $\i$-categories also admits a
symmetric monoidal structure with unit the $\i$-category $\ispec$ of
spectra \cite[\S 6.3.1]{HA}.  There is a map of commutative algebra
objects $\spaces\to\spaces_*\to\spectra$ of $\PR^\L$, where
$\spaces_*$ is the $\i$-category of pointed spaces.  Both of these
maps are ``localizations'' in the sense that the endofunctors
$(-)\otimes_{\spaces} \spaces_*$ and $(-)\otimes_{\spaces} \spectra$
of $\Mod_{\spaces}$ are idempotent: clearly
$\spaces_*\otimes_{\spaces} \spaces_*\simeq\spaces_*$, and the same is
true for $\spectra$ since $\spectra\simeq\spaces_*[\Sigma^{-1}]$.

Given an arbitrary $\i$-operad $\aO^\otimes$, a (not necessarily
small) $\aO$-monoidal $\i$-category is an object of
$\Alg_\aO(\widehat{\Cat}_\i)$, and an $\aO$-monoidal presentable
$\i$-category is an object of $\Alg_{\aO}(\PR^{\L})$.  In particular,
associated to an $\bE_n$ algebra $R$ in a symmetric monoidal
$\i$-category $\aC$, there is an $\i$-category $\Mod_R$ of right
$R$-modules.  A central theorem in the subject (verifying a conjecture
of Mandell) is that an $\bE_n$-algebra $R$ induces an
$\bE_{n-1}$-monoidal structure on $\Mod_R$ with unit $R$ such that the
tensor product commutes with colimits in each variable; moreover, the
functor $R \mapsto \Mod_R$ from $\bE_n$-algebras to
$\bE_{n-1}$-monoidal $\i$-categories is
fully-faithful~\cite[6.3.5.17]{HA}.  Furthermore, in an
$\bE_n$-monoidal $\i$-category $\aC$, for $m \leq n$ the map of
$\i$-operads $\bE_m \to \bE_n$ implies that we have an $\bE_{n-m}$
monoidal $\i$-category $\Alg_{\bE_{m}/\bE_n}(\aC)$ of $\bE_m$ algebra
objects.

\section{Parametrized spaces and spectra}
\label{sec:wirthmuller-contexts}

In this section, we review explicit $\i$-categorical models of
parametrized spaces, spectra, and $R$-modules for an $\bE_n$-ring
spectrum $R$ from~\cite{ABGHR, ABG, ABGHR1}.  Our main goal is to
provide the necessary background for Section~\ref{sec:twisted}.  In
Section~\ref{sec:param}, we develop the general theory of arbitrary
$\i$-categories parametrized over arbitrary $\i$-topoi, which
specializes to the definitions given in this section.

We begin with the definition of parametrized spaces.

\begin{definition}
Let $S$ be a Kan complex, which we view as an object of the $\i$-category of $\i$-groupoids.
Define the $\i$-category of spaces over $S$ as the $\i$-category
\[
\spaces_{/S} = \Fun(S^{\op}, \spaces)
\]
of presheaves of spaces on $S$, and the $\i$-category of pointed
spaces over $S$ (or ex-spaces) as the $\i$-category 
\[
(\spaces_{/S})_* = \Fun(S^{\op},\spaces)_*\simeq\Fun(S^{\op},\spaces_*) =
(\spaces_*)_{/S}
\]
of presheaves of pointed spaces on $S$.
\end{definition}

\begin{remark}
The $\i$-category of $\i$-categories admits an autoequivalence $(-)^{\op}$ which sends $\C$ to $\C^{\op}$, which when restricted to the full subcategory of $\i$-groupoids comes equipped with a natural equivalence $(-)^{\op}\to\id$.
This follows, for instance, from the fact the $\i$-category of $\i$-groupoids can be modeled by the full subcategory of simplicially-enriched categories consisting of the simplicial groupoids. In particular, any $\i$-groupoid $S$ comes equipped with a canonical involution $S^{\op}\simeq S$, so that $\Fun(S^{\op},\spaces)\simeq\Fun(S,\spaces)$.
For this reason, we will sometime ignore the $(-)^{\op}$ and write $\spaces_{/S}\simeq\Fun(S,\spaces)$.
The same goes for presheaves on $S$ valued in an arbitrary $\i$-category $\C$, such as pointed spaces or spectra.
\end{remark}

%We have the analogous definition of parametrized spectra, as presheaves of spectra on $S$.

\begin{definition}
Let $S$ be a Kan complex.  The $\i$-category of spectra over $S$
is defined to be the $\i$-category 
\[
\ispec_{/S} = \Fun(S^{\op}, \ispec)
\]
of presheaves of spectra on $S$.
\end{definition}

We would expect that $\i$-category $\ispec_{/S}$ of spectra
parametrized over $S$ can be understood as the stabilization of the
$\i$-category $\spaces_{/S}$ (or $(\spaces_{/S})_*$) of (pointed) spaces
parametrized over $S$.  Indeed, we have the following proposition.

\begin{proposition}
Let $S$ be a space and let $\aC$ be a presentable $\i$-category. 
Then we have a natural equivalence
\[
\Stab(\aC_{/S})\simeq\Stab(\aC)_{/S}
\]
of stable presentable $\i$-categories.
Moreover, if $\aC$ is an $\cO$-monoidal presentable $\i$-category for
some $\i$-operad $\cO$, then this equivalence extends to an
equivalence of $\cO$-monoidal stable presentable $\i$-categories. 
\end{proposition}

\begin{proof}
First recall that the functor $\Stab(-) \simeq (-)\otimes \ispec$ is a
symmetric monoidal localization of $\Pr^\L$ whose essential image is
precisely the full subcategory of stable presentable $\i$-categories. 
Thus
\[
\Stab(\aC_{/S})\simeq (S\otimes\aC)\otimes\spectra\simeq
S\otimes(\aC\otimes\spectra)\simeq\Stab(\aC)_{/S}. 
\]
The final assertion follows from the fact that $(-)_{/S}$ is also 
symmetric monoidal, which we prove as proposition~\ref{prop:paramsym} 
below.  
\end{proof}

Finally, we describe parametrized module spectra.  

\begin{definition}\label{defn:modbundle}
Let $R$ be an $\bE_1$-ring spectrum, let $S$ be a space, and let
$\Mod_R$ denote the stable presentable $\i$-category of right $R$-modules.
The $\i$-category of parametrized $R$-module spectra (i.e., bundles of
$R$-modules) over $S$ is the stable presentable $\i$-category
$(\Mod_R)_{/S}$.  
\end{definition}

In practice, however, $R$ is often more than an $\bE_1$-ring spectrum;
rather, it may be an $\bE_n$-ring spectrum for some $1\leq
n\leq\infty$.  In this case, the category of parametrized module
spectra inherits a multiplicative structure where the product
combines the product on $R$-modules with the diagonal map on the base
space.

\begin{proposition}
Let $R$ be an $\bE_n$-ring spectrum, $n>0$, and let $S$ be a space.
Then the $\i$-category $(\Mod_R)_{/S}$ of parametrized $R$-module spectra
over $S$ is the underlying $\i$-category of an $\bE_{n-1}$-monoidal
stable presentable $\i$-category $(\Mod_R)_{/S}^\otimes$. 
\end{proposition}

\begin{proof}
Let $\Mod_R$ denote the $\i$-category of right $R$-modules, which
is an $\bE_{n-1}$-algebra object of $\PrL$ and in particular an
$\bE_{n-1}$-monoidal $\i$-category.  Then
\[
\Fun(S^{\op},\Mod_R)^\otimes=\Fun(S^{\op},\Mod_R^\otimes)\underset{\Fun(S^{\op},\N(\Gamma))}{\times}\N(\Gamma)
\]
is an $\bE_{n-1}$-monoidal $\i$-category such that the underlying
$\i$-category 
\[
(\Mod_R)_{/S}=\Fun(S^{\op},\Mod_R)
\]
is stable and presentable.
By construction, the operations in $\Fun(S^{\op},\Mod_R)^\otimes$ are
computed pointwise, so that the tensor product commutes with colimits
in each variable.  Hence
$(\Mod_R)_{/S}^\otimes=\Fun(S^{\op},\Mod_R)^\otimes$ is an
$\bE_{n-1}$-monoidal stable presentable $\i$-category.
\end{proof}

\section{Twisted cohomology theories and the twisted Umkehr map}\label{sec:twisted}

Let $f\colon X\to B$ be a bundle of smooth manifolds, and let $T_{f}$ be
the bundle of tangent vectors along the fiber, say of
dimension $d$.  The Pontryagin-Thom construction gives rise to a
stable map 
\begin{equation}\label{eq:10}
   \PT (f)\colon    \splus B \xra{} X^{-T_{f}},
\end{equation}
which we'll call the ``Pontryagin-Thom transfer'' associated to $f$.  
In the presence of a Thom isomorphism 
\[
     R^{*} (\splus X) \xra{} R^{*-d} (X^{-T_{f}}) 
\]
one then has an Umkehr homomorphism
\[
     R^{*} (\splus X) \xra{} R^{*-d} (\splus B).
\]
A number of contexts, ranging from the $K$-theoretic analysis of
anomalies in string theory to the Umkehr map in Grojnowski's
equivariant elliptic cohomolgy, with its role in the proof of Witten's
rigidity theorems and the derivation of the Kac-Weyl character formula
\cite{fw:asd,MR2438341,Rosu:Rigidity,Ando:AESO,MR3230851}, suggest the
following generalization.   Suppose that $R$ is an $\einfty$ ring
spectrum, and let $\alpha\colon B \to \Pic_{R}$ classify an invertible
$R_{B}$-module.  One can ask for a ``twisted Umkehr map'' 
\[
     R^{*+T_{f}-\alpha} (\splus X) \xra{} R^{*-\alpha} (\splus B).
\]
For example if $-T_{f}+\alpha$ is null homotopic, then a choice of
null-homotopy determines a map
\[
    R^{*} (\splus X) \xra{} R^{*-\alpha} (\splus B), 
\]
and if $\alpha$ itself is null, one recovers the Umkehr map.

In this section we construct such a twisted Umkehr map.  The key
points are:
\begin{enumerate}
\item The Pontryagin-Thom transfer map \eqref{eq:10} arises from a map
of spectra  
\[
       \bS_{B} \xra{} D_{B} X
\]
over $B$ by pushforward along the map $p\colon B \to \ptspace$,
and the twisted Umkehr map arises by smashing with the
bundle classified by $\alpha$
\[
      \alpha  \xra{} \alpha \Smash_{B} D_{B} X
\]
and then pushing forward along $p$, 
\[
    p_{!} \alpha \xra{} p_{!} (\alpha \Smash_{B} D_{B} X).
\]
\item Identifying $p_{!} D_{B} (X) \heq X^{-T_{f}}$ requires a
parametrized version of Atiyah duality 
\[
     D_{B} X \heq \bS_{B}^{-T_{f}}
\]
from which one concludes that 
\[
p_{!}D_{B} X \heq X^{-T_{f}}.
\]
\end{enumerate}

\subsection{The Becker-Gottlieb transfer}\label{sec:bg}

We begin by recalling from \cite[\S 15.3]{MS} the construction of the
Becker-Gottlieb transfer in the setting of parametrized homotopy
theory.  The transfer map arises from the categorical trace
associated to the diagonal map of a stably dualizable space $X$.
Specifically, we have the composite
\[
\xymatrix{
\bS \ar[r]^-{\eta} & X \sma DX \ar[r] & DX \sma X \ar[r]^-{\id \sma
  \Delta} & DX \sma X \sma X \ar[r]^-{\epsilon \sma \id} & \bS \sma X \simeq
X,
}
\]
where $\eta$ and $\epsilon$ are the coevaluation and evaluation of the duality.
These transfer maps satisfy a series of compatibility relations, see
\cite[15.2.4]{MS}; the required conditions on the triangulation
of the homotopy category hold here, either by comparison or as can be
shown directly.

The key observation about duality in the parametrized setting is that
we can characterize dualizability fiberwise because the smash product is computed pointwise.

\begin{lemma}
Let $B$ be a space and $X\in\Fun(B^{\op},\spectra)$ a parametrized
spectrum.  Then $X$ is dualizable if and only if for each $b \in B$,
the value $X_b$ of $X$ at $b$ is a dualizable spectrum.
\end{lemma}

In particular, given a map $f\colon X \to B$ of spaces with stably dualizable
(homotopy) fibers, such as a proper fibration, by adjoining a disjoint basepoint we get a diagonal map on $\Sigma^{\infty}_B X_+$ and so a transfer map
\[
\bS_B \to \Sigma^\infty_B X_+.
\]
Pushing forward along the map $B \to *$ now yields the classical
transfer map
\[
\splus B \to \splus X.
\]

Note that we can easily recover Dwyer's generalization of the
transfer \cite{Dwyer}.  Specifically, let $R$ be an
$\bE_\infty$-ring spectrum and suppose that $f \colon E \to B$ has
homotopy fiber $F$ such that $R \sma \Sigma^\infty_+ F$ is a dualizable
object in the category of $R$-modules.  Then
the construction of the transfer in this setting gives rise to an
$R$-module transfer map 
\[
R \sma \splus B \to R \sma \splus E.
\]

Summarizing, we obtain the following result of Dwyer.

\begin{proposition}
Let $R$ be an $\bE_\infty$-ring spectrum and let $f \colon E \to B$ be a
map of spaces with homotopy fiber $F$ such that
$R \sma \Sigma^{\infty}_+ F$ is a dualizable object in $\Mod_R$.  Then
the diagonal map on $E$ gives rise to a map of $R$-modules over $B$
\[
R_B \to R \sma \Sigma^{\infty}_B E_+
\]
such that the pushforward along $B \to *$ is the $R$-module transfer
\[
R \sma \splus B \to R \sma \splus E.
\]
\end{proposition}

\subsection{Duality and the Pontryagin-Thom map}
\label{sec:dual-pontrj-thom}

The simple construction of the transfer map of the previous section
does not reveal the powerful relationship of the transfer to geometry
and analysis.  That relationship is mediated by the Umkehr map and
geometric constructions of the dual.

As above, recall that $\spaces$ denotes the $\i$-category of spaces and
$\spectra$ denotes the $\i$-category of spectra.  We let $\bS$ denote
the sphere spectrum and, for a space $X$, we write $DX$ for the
Spanier-Whitehead dual 
\[
   DX = F(\splus X,\bS)
\]
of the spectrum $\splus X$. We may regard 
\[
D\colon \spaces^{op} \to \spectra
\]
as a presheaf of spectra $\spaces$.  Applying $D$ to the unique map of spaces
$p\colon X\to \ptspace$
gives a map of spectra 
\[
\phi (X)\colon \bS\to \splus X;
\]
which we may regard as a functor
\[
    \phi\colon  \spaces^{\op} \to \spectra_{\bS_{/}},
\]
where $\spectra_{\bS_{/}}$ denotes the category of spectra under $\bS$.
If $X$ is a compact manifold, then the Pontryagin-Thom construction and
Aityah duality give a wonderful formula for this map. 
Take an embedding $X \to \bR^{N}$ with normal bundle $\nu_{X}$, and
form the Pontryagin-Thom construction (collapse to a point the
complement of a tubular neighborhood of $X$ in $\bR^{N}$)  to get a
map  
\[
    S^{N} \to X^{\nu_{X}}.
\]
Desuspending $N$ times gives a stable map, the 
Pontryagin-Thom map, 
\[
\PT\colon     \bS \to X^{-T},
\]
where $T$ denotes the tangent bundle of $X$.

\begin{proposition}
There is an equivalence $X^{-T} \to D (X)$ such that the diagram 
\[
\xymatrix{
{\bS}
 \ar[r]^-{\PT}
 \ar[dr]_{\phi (X)}
& 
{X^{-T}} 
 \ar[d]
\\
& 
{D (X)}
}
\]
commutes up to homotopy.  
\end{proposition}

Suppose that $R$ is an $\einfty$ ring spectrum and we have a Thom
isomorphism  
\[
R^{*} (X_{+})\iso R^{*-d} (X^{-T}).
\]
The Umkehr map associated to the map $\pi_{X}\colon X\to \ptspace$ and
the Thom isomorphim is the composition
\[
\xymatrix{
{R^{*} (X_{+})}
 \ar[r]
&
{R^{*-d} (X^{-T})}
 \ar[dr]^{PT^{*}}
 \ar[d]_{\iso}
\\
&
{R^{*-d} (DX)}
 \ar[r]_{\phi (X)}
&
{R^{*-d} (S^{0}).}
}
\]

Now suppose that $f\colon X \to B$ is a smooth and proper family of
manifolds over $B$; in other words, for each $b\in B$, the fiber $X_b$
is a smooth and proper manifold which varies continuously over $B$ in
the sense that $X$ is classified by a map $B^{\op} \to\Man$, where
$\Man$ denotes the $\i$-category of smooth and proper manifolds and
diffeomorphisms.  Here $\Man$ can be described as the coherent nerve
of the ordinary category of smooth compact manifolds and
diffeomorphisms.
%\textbf{We could also use smooth maps and then restrict to the
%isomorphisms; whichever seems a better choice for the category of
%manifolds is fine with me.}
Observe that $\Man$ is an $\infty$-groupoid which, when regarded as a space, decomposes as the sum
\[
\Man\simeq\coprod_{[M]}B\mathrm{Diff}(M),
\]
indexed over diffeomorphism classes of smooth
manifolds, where $\mathrm{Diff}(M)$ denotes the (topological) group of
diffeomorphisms of a representative $M$ for the class $[M]$.
If $B=BG$ is
connected, then this amounts to an $A_\infty$-map
$G\to\mathrm{Diff}(M)$ for some smooth and proper manifold $M$, and 
$X\simeq M/G$ is the homotopy quotient of $M$ by its $G$-action.  

%\textbf{I'm not sure we've introduced the actors here carefully
%enough, but I like the way this goes.}

The composition 
\[
    B^{\op} \xra{} \Man \xra{} \spaces \xra{\phi} \spectra_{\bS/}
\]
gives an object of $(\spectra_{\bS/})_{/B} \heq
(\spectra_{/B})_{\bS_{B}/}$, that is, a $B$-parametrized spectrum
\[
     \phi_{X/B}\colon \bS_{B} \xra{} D_{B} (X),
\]
under
$\bS_{B}$
whose value at $b \in B$ is 
\[
     \phi (X_{b})\colon \bS  \xra{} D (X_{b}).
\]
This is the fiberwise dual of the map $X_{b}\to \ptspace$.

Pushing forward along $p\colon B\to \ptspace$, we obtain a map 
\[
  \splus B \heq p_{!} p^{*} \bS  \heq  p_{!}\bS_{B} \xra{} p_{!}D_{B} (X).
\]
To identify $p_{!}D_{B} (X)$ with $X^{-T_{f}}$, and so obtain the
Pontryagin-Thom transfer map
\eqref{eq:10} as advertised, we need a parametrized form of Aityah duality.

\subsection{Parametrized manifolds and fiberwise Atiyah-Milnor-Spanier
duality}\label{sec:geom}

If $X$ is a space, we continue to write $DX = F (\splus X,
\bS)$ for the dual of $\splus X$, and, if $Z$ is a spectrum, we write
$DZ$ for the dual $F 
(Z,\bS)$.  

The usual construction of the Atiyah duality
equivalence
$
       X^{-TX}\to DX
$
does not have attractive functoriality and naturality properties, and
so it is not straightforward to assemble the fiberwise maps.  We give
a new approach to work of \cite{Hu} and \cite{MS}, which shows
that the fiberwise dual of a bundle of manifolds $X\to B$ can be 
calculated by a parametrized version of the Pontryagin-Thom
construction. 

Again let $f\colon X\to B$ be a continuous family of smooth and proper
manifolds,  classified by a map $B\to\Man$, the $\i$-category of
smooth and proper manifolds.  Let
\[
T_f=T_{X/B}
\]
denote the bundle of tangents along the fiber of $f\colon X\to B$.  We write
$\bS_X^{T_{X/B}}$ or $\bS_X^{T_{f}}$ for the suspension spectrum of
the associated sphere bundle; it is a bundle of spectra over $X$.  

We will show that $f_{!} \Sigma_B^V \bS_{B}^{-T_{f}}$ is naturally equipped
with an equivalence to $D_{B} X$, where here $V$ is a Euclidean
space in which the fiber $F$ embeds.  Our approach relies on the observation
(which we learned from \cite{Hu}) that the suspension spectrum of
the cofiber 
\[
   C_f = \hocofib(X\times_{B}X - \Delta \xra{} X\times_{B}X) \htp (X\times_{B}X/X\times_{B}X - \Delta),
\]
gives a model for $\bS_{X}^{T_{X}}$, where $\Delta$ denotes the image of the diagonal $X\to X\times_B X$ and we regard $X \times_{B} X$
as a space over $B$ via the projection onto the first coordinate; this
makes the diagonal into a map over $B$.

We begin by considering the case in which $B = *$.  Then the classical
observation (e.g., see~\cite[\S10]{MilnorStasheff} or
\cite[\S12]{Bredon}) that for a compact 
manifold $M$, the normal bundle of the diagonal embedding
$M \to M \times M$ is homeomorphic to the tangent bundle
of $M$ implies the following lemma.

\begin{lemma}
The (homotopy) cofiber $C_f = \Sigma_X^{\infty} (X \times X /
X\times X - \Delta)$ is equivalent to the tangent sphere bundle
$\bS_X^{T}$.
\end{lemma}

Since the case when $B = *$ describes the fiber in the general case,
this observation now yields the following general description.

\begin{corollary}
The (homotopy) cofiber $C_f = \Sigma_X^{\infty} (X \times_B X /
X\times_{B}X - \Delta)$ is equivalent to the tangent sphere bundle
$\bS_X^{T_{X/B}}$.
\end{corollary}

We now turn to analyzing the dual of $C_f$.  Again, we begin by studying
the case in which $B = *$, so that $X$ is just a smooth and proper manifold. We choose a smooth embedding of $X$ in 
a Euclidean space $V$ with normal bundle $\nu$.

\begin{lemma}
For $f \colon X \to *$, the parametrized spectrum $C_f$ is dualizable
and in fact invertible, with inverse given by
$\Sigma_X^{-V} \bS^{\nu}$.
\end{lemma}

\begin{proof}
This follows from the evident equivalence
\[
\bS_X^{T_{X/B}} \sma \bS_X^{\nu} \to S^V,
\]
induced from the fact that $\tau \oplus \nu$ is the trivial bundle.
\end{proof}

Since dualizability is detected fiberwise, this has as an immediate
corollary invertibility for the case of general $f \colon X \to B$
with manifold fibers.

\begin{corollary}
The parametrized spectrum $C_f$ is dualizable and in fact invertible.
\end{corollary}

\begin{remark}
The central observation of \cite{Hu} is that the inverse of $C_f$
gives a model for a dualizing complex in the sense of
Grothendieck, adapted to the Wirthm\"uller setting.  
\end{remark}

We now exhibit a natural map 
% \[
%   \theta \colon f_{!}D_{X} (\Sigma^{\infty}_X X\times_{B}X/X\times_{B}X-\Delta (X)) \to D_{B} (X)
% \]
\[
  \theta \colon f_{!}D_{X} C_{f} \to D_{B} X
\]
and then show that it is an equivalence.  

\begin{construction}
The natural evaluation map
\[
C_{f} \Smash_{X} D_{X}C_{f} \to \bS_{X},
\]
composed with the map to the cofiber
\[
   \sinf_{X} (X\times_{B}X) \to C_{f},
\]
produces a natural composite
\[
  \sinf_{X}(X \times_{B}X) \Smash_{X} D_{X} C_{f} \to \bS_{X}.
\]
We have canonical and natural equivalences $\sinf_{X}
(X\times_{B}X) \htp 
f^{*}\sinf_{B}X_{+}$ and 
$f^{*}\bS_{B}\heq \bS_{X}$, and so we may rewrite the map as 
\[
f^{*}\sinf_{B}X_{+} \Smash_{X} D_{X} C_{f} \to f^{*}\bS_{B}.
\]
Passing to adjoints, we have
\[
D_X C_{f}  \to F_X( f^* \sinf_B X_+, f^* \bS_{B}).
\]
Next, using the natural equivalence $F_X( f^* \sinf_B X_+, f^* \bS_{B}) \htp f^* F_B(\sinf_B X_+, \bS_{B})$, we obtain the map 
\[
D_X C_{f}  \to f^* D_B X, 
\]
and finally applying the $(f_!, f^*)$ adjunction yields the desired
map $\theta$.
\end{construction}

\begin{proposition}
The natural map 
\[
\theta \colon f_! D_X (X \times_B X / X \times_B X - \Delta) \to D_B X
\]
is an equivalence.
\end{proposition}

\begin{proof}
Since equivalences of parametrized spectra are detected fiberwise, it
suffices to restrict to the fiber over each point $b \in B$.
Equivalently, we can assume that $B = *$.  In this case, the map
reduces to 
\[
f_! D_X (X \times X / X \times X - \Delta) \to DX.
\]
By the discussion in~\cite[\S 19.6]{MS}, $D_X(X \times X / X \times X
- \Delta) \htp \Sigma_X^{-V} \bS_X^{\nu}$, and since
$f_! \Sigma_X^{-V} \bS_X^{\nu} \htp \Sigma^{-V} X^{\nu}\htp
X^{-T_{X}}$, abstractly Atiyah duality implies the equivalence we
want. 

We need to check that the map in question is homotopic to the standard
map inducing the Atiyah duality equivalence.  First, recall that the
evaluation map 
\[
\epsilon \colon X^{\nu} \sma X_+ \to S^V
\]
which induces the Atiyah duality equivalence $X^{-T_{X}}\to D X$ 
is induced from the
Pontryagin-Thom construction applied to the composite of the diagonal
$X \to X \times X$ and the zero section $X \times X \to \nu \times
X$ (note that the normal bundle of the compsition of these two
embeddings is trivial).   Next, the composite
\[
X \times X \to X \times X / X \times X
- \Delta
\]
is a model for the Pontryagin-Thom collapse map
associated to the embedding of the diagonal.  Therefore, under the
equivalences $D_X(X \times X / X \times X
- \Delta) \htp \Sigma_X^{-V} \bS_X^{\nu}$ (which is not natural)
and $\sinf_{X} (X \times X) \htp f^* \Sigma_+X$ (which is natural), the map 
\[
f_! (\sinf_{X}(X \times X) \Smash_{X} D_{X}
(X\times X/X\times X-\Delta)) \to \bS
\]
is homotopic to the map
\[
f_! (f^* \Sigma_+ X \Smash_{X} \Sigma_X^{-V} \bS_X^{\nu}) \to \bS
\]
which, when expressed as
\[
\Sigma_+ X \sma \Sigma_X^{-V} X^{\nu} \to \bS,
\]
is the usual Atiyah duality map.
\end{proof}

\begin{corollary}
A choice of equivalence $C_{f}\heq \bS_{X}^{T_{f}}$ induces an
equivalence of spectra over $B$ 
\[
f_!\bS_{X}^{-T_{f}}\simeq D_B X.
\]
Pushing this equivalence forward along the map $p\colon B\to \ptspace$
yields an equivalence of spectra  
\[
   p_{!}D_{B} X \heq X^{-T_{f}}.
\]
\end{corollary}

\begin{proof}
The parametrized spectrum $\bS_{X}^{-T_{f}}$ is classified by a map
\[
  X^{\op} \too \Pic_{\bS},
\]
and the Thom spectrum $X^{-T_{f}}$ is just the colimit along $pf\colon
X\to\ptspace$; that is
\[
  X^{-T_{f}} =  (pf)_{!} = p_{!}f_{!}\bS^{-T_{f}}.
\]
\end{proof}

\begin{definition}\label{def-fww}
The \emph{parametrized Pontryagin-Thom transfer} is the composition 
\[
    \PT_{/B} (f)\colon \bS_{B} \xra{\phi_{/B}} D_{B} X \simeq
    f_{!}\bS_{X}^{-T_{f}}.  
\]
\end{definition}

In conclusion, we have the following.

\begin{proposition} \label{t-pr-umkehr-fiberwise}
Applying the push-forward $p_{!}$ to the parametrized Pongrjagin-Thom
transfer 
\[
   \PT_{/B} (f)\colon \bS_{B}\to      f_{!}\bS_{X}^{-T_{f}}
\]
gives a map
\begin{equation}\label{eq:14}
  \splus B \to X^{-T_{f}}.
\end{equation}
\end{proposition}

\begin{definition} \label{def-PT-tr}
The \emph{Pontryagin-Thom transfer map} associated to $f\colon X\to B$ is
the map 
\[
\PT (f)\colon \splus B \to X^{-T_{f}}
\]
given by proposition \ref{t-pr-umkehr-fiberwise}.  If $R$ is a ring
spectrum, then a Thom isomorphism $R^{*} (\splus X)\heq R^{*-d}
(X^{-T_{f}})$ and the map induced by \eqref{eq:14} determine
an \emph{Umkehr} map
\[
    R^{*} (\splus X) \heq R^{*-d} (X^{-T_{f}}) \to R^{*-d} (\splus B).
\]
\end{definition}

\begin{remark} \label{rem:embedding}
We also expect to have Umkehr maps arising from embeddings of
manifolds.  Let $j \colon W\to M$ be an embedding of manifolds with
normal bundle $\nu$, and let $p\colon M\to \ptspace$ be the map to a
point.  To obtain a similar view of the Umkehr map $j$, we need to
realize the geometric Pontryagin-Thom map
\[
\Sigma^\i_+ M \to\Sigma^\i W^{\nu}
\]
as $p_{!}t,$ where $t$ is a map of spectra over $M$.  Now if $S^{\nu}$
is the parametrized space associated to $\nu$, then
\[
       W^{\nu} \heq p_{!}j_{!} S^{\nu},
\]
and this suggests that the map we seek is a suspension of a map of the
form
\[
     \alpha\colon S^0_{M} \to j_{!} S^{\nu}.
\]
The required map is constructed by May and Sigurdsson \cite[18.6.3,
18.6.5]{MS}.
\end{remark}

\subsection{Twists}
\label{sec:twists}

The Pontryagin-Thom transfer map $\PT (f)$ of Proposition
\ref{t-pr-umkehr-fiberwise} arises 
%\textbf{By which we mean the
%generalized PT map associated to X to B, which we never quite define,
%and this is an issue at the end of the preceding section} 
from a map of specra parametrized by $B$, and so we can twist it.  
We begin by recalling the  
notion of a twist (see also \cite{ABG}).

\begin{definition}
Let $R$ be an $\bE_{n+1}$-ring spectrum.  A {\em twist} is a map
\[
        \alpha \colon B \to \Pic_R \to \Mod_R
\]
classifying a bundle of invertible $R$-modules over $B$.  Notice that
the $\i$-category $\spaces_{/\Pic_R}$ of twists is $\bE_n$-monoidal.
\end{definition}

Given a twist $\alpha \colon B \to \Pic_R \to \Mod_R$ and a
generalized Umkehr map 
\[
R_B\longrightarrow X
\]
we can twist the map by
fiberwise smash to obtain a map
\[
R_B \sma_{R_B} \alpha \to X \sma_{R_B} \alpha. 
\]

We can interpret this using our definition of twisted cohomology.  Let
$p \colon B \to *$ denote the terminal map.  As explained
in \cite{ABGHR,ABG}, $M\alpha = p_{!} (R_{B}\Smash_{R_B} \alpha)$ is
the generalized Thom spectrum of the map $\alpha$; by definition, its
$R$-cohomology is the $\alpha$-twisted $R$-cohomology of $B$.
\[
     R^{k-\alpha} (B) = \pi_{0} \map_R (M\alpha,\Sigma^{k} R). 
\]

\begin{remark}
When $R$ is the K-theory spectrum (real or complex), the question
arises of comparing our version of twisted $K$-theory to the
Atiyah-Segal construction of twisted $K$-theory in terms of Fredholm
operators.  In \cite[\S 5]{ABG}, we interpret the Atiyah-Segal
construction as associating to a twist $f \colon X \to B\GL{A}$ the
spectrum 
\[
    \Gamma_X (f) \htp \ispec_{/ X}(\bS_{X},f),
\]
i.e., the spectrum of sections of $f$.  Further, we explain how this
spectrum is equivalent to the Thom spectrum functor applied to the
twist $-f$ (the image under the involution $-1 \colon B\GL{A} \to
B\GL{A}$).

One might worry however that the geometric aspects of the Atiyah-Segal
construction associate to a twist $X \to K(\bZ/2,2)$ a composite other
than the that induced by the inclusion $K(\bZ/2,2) \to B\GL{KO}$, and
so there could be a potential discrepancy.  But in \cite{AGG}, the
third author (along with Antieau and Gomez) prove that up to homotopy
any map $j \colon K(\bZ/2,2) \to B\GL{KO}$ is either trivial or the
canonical inclusion (and similarly for $KU$).
\end{remark}

\begin{remark}
There are many other interesting examples of twisted cohomology theories.
\begin{enumerate}

\item
The spectrum $tmf$ of topological modular forms comes equipped with a
canonical map $K(\bZ,4)\to\Pic_{tmf}$; see \cite{ABG} for the details
of construction of this map.  Note that in \cite{ABG} we used the
connected component $BGL_1(tmf)$ of the identity in the grouplike
$\bE_\infty$-space $\Pic_{tmf}$, instead of the whole Picard space,
which is of course sufficient since $K(\bZ,4)$ is connected.

\item
Another ``form of elliptic cohomology'' is the algebraic K-theory $K(ku)$ of the connective topological K-theory spectrum $ku$.
It is also equipped with a map $K(\bZ,4)\to BGL_1(K(ku))\simeq\Pic_{K(ku)}$ constructed as follows: delooping, it suffices to produce an $\mathbb{A}_\infty$-map $K(\bZ,3)\to GL_1(K(ku))$.
Using the composition $BGL_1(ku)\to BGL(ku)\to\bZ\times
BGL(ku)^+\simeq\Omega^\infty K(ku)$, which is a map of
$\bE_\infty$-spaces for the multiplicative structure on $\Omega^\infty
K(ku)$, we obtain this map by delooping the $\bE_\infty$-map
$K(\bZ,2)\to GL_1(ku)$. 

\item
The family of $\mathbb{E}_\infty$-ring spectra (defined for each prime
$p$ and positive integer $n$) studied by C. Westerland admit twists by
$K(\bZ_p,n)$.  These spectra $R_n$ are defined as the homotopy fixed
points $E_n^{h S\mathbb{G}_n^\pm}$ of the Lubin-Tate spectra $E_n$,
and admit Snaith-style presentations of the form $R_n\simeq
L_{K(n)}\Sigma^\infty_+ K(\bZ_p,n+1)[\rho^{-1}]$.  There therefore
come equipped with canonical $\bE_\infty$-maps
$K(\bZ_p,n+2)\to\Pic_{R_n}$.
\end{enumerate}
\end{remark}

The monoidal structure on the category of twists gives rise to a
product in twisted cohomology:

\begin{theorem}
Let $R$ be an $\bE_{n+1}$ ring spectrum.  For any space $X$, the
$\bE_n$ monoidal structure on $\spaces_{/\Pic_R}$ gives rise to a product map
\[
R^\alpha(X) \otimes R^\beta(X) \to R^{\alpha+\beta}(X).
\]
\end{theorem}

We now specialize to the geometric example considered in
section~\ref{sec:geom}.  
We continue to let $f\colon X\to B$ denote a family of compact manifolds
over a space $B$, and to  write
$p$ for the map $B \to \ptspace$.  Given a twist $\alpha$, we can 
form the map of $R$-modules over $B$  
\[
   \PT_{/B}(f) \Smash_{R_B} \id \colon R_{B} \Smash_{R_B} \alpha \to f_{!}\bS_{X}^{-T_{f}} \Smash_{R_B} \alpha.
\]
Applying the pushforward $p_{!} \colon (\Mod_R)_{/B} \to \Mod_R$
gives rise to the twisted Pontryagin-Thom transfer.

\begin{definition}
The twisted Pontryagin-Thom transfer map is defined as:
\[
\PT (f,\alpha) = p_{!} \Bigl(\PT_{/B}(f) \Smash_{R_B} \id \colon R_{B} \Smash_{R_B} \alpha \to f_{!}\bS_{X}^{-T_{f}} \Smash_{R_B} \alpha \Bigr).
\]
\end{definition}

As for $f_{!}\bS_{X}^{-T_{f}} \Smash_{R_B} \alpha$, the projection
formula yields
\begin{align*}
f_{!}\bS_{X}^{-T_{f}} \Smash_{R_B} \alpha & \heq 
f_{!} (R_{X} \sma_{\bS_X} \bS_X^{-T_{f}}) \Smash_{R_B} \alpha \\
 & \heq f_{!} ((R \sma_{\bS_X} \bS_X^{-T_{f}}) \Smash_{R_X} f^{*}\alpha),
\end{align*}
and so 
\[
  p_{!} (f_{!}\bS_{X}^{-T_{f}} \Smash_{R_B} \alpha)  \heq p_{!}f_{!}
  ((R_{X} \sma_{\bS_X} \bS_X^{-T_{f}})\Smash_{R_X} f^{*}\alpha) =
  X^{-T_{f}+ \alpha f}
\]
is the $R$-module Thom spectrum whose $R$-module cohomology is the
cohomology of $X$, twisted by the sum of 
\[
   X \xra{-T_{f}} \Pic_{\bS} \xra{} \Pic_{R} \to \Mod_R
\]
and 
\[
  X  \xra{f} B \xra{\alpha} \Pic_{R} \to \Mod_R.
\]
That is, we have the following.  
\begin{proposition}\label{prop:twisty}
Let $f\colon X\to B$
be a family of compact manifolds, and let $\alpha \colon
B\to \Pic_R \to \Mod_R$ be a parametrized invertible $R$-module over
$B$.  Then we have an equivalence of $R$-modules
\[
  X^{-T_{f} + \alpha f} \heq p_{!} (f_{!}\bS_{X}^{-T_{f}} \Smash_{R_B}
	\alpha),
\]
and so the twisted Pontryagin-Thom transfer $\PT(X,\alpha)$ may be viewed 
as a map of $R$-modules 
\[
    \PT (X,\alpha) \colon  B^\alpha \xra{} X^{-T_{f} + \alpha f}.
\]
\end{proposition}

Passing to $R$-cohomology gives a twisted Umkehr map
\[
      R^{*+Tf-\alpha} (X) \xra{} R^{*-\alpha} (B).
\]
We close with an example, motivated by \cite{MR2438341,fw:asd}.  Suppose that $\alpha$ makes the diagram 
\[
\xymatrix{
X \ar[d]_{f} \ar[r]^-{-Tf} & \Pic_{\bS} \ar[d] \\
B \ar[r]^-{\alpha} & \Pic_{R} \\
}
\]
commute in the homotopy category.  A choice of homotopy between the
two compositions determines an equivalence of
$R$-modules 
\[
   X^{-Tf+\alpha f} \heq \splus X\Smash R,
\]
so the twisted Pontryagin-Thom transfer map takes the form 
\[
    \PT(X,\alpha) \colon B^{\alpha} \xra{} \splus X \Smash R,
\]
and passing to $R$-cohomology yields a twisted Umkehr map 
\[
     R^{*} (X) \xra{} R^{*-\alpha} (B).
\]

The following instance of this construction was described in our paper
\cite{ABG}, and was inspired by the work of Freed and Witten and
Carey and Wang \cite{fw:asd,arixv:math/0507414,MR2438341}. Let $j\colon D \to  X$ be an embedded submanifold, let $\nu$ be the normal
bundle of $j$, and suppose that $D$ carries a
complex vector bundle $\xi.$

If $\nu$ carries a $\Spinc$-structure, then we can
form the $K$-theory push-forward
\[
j_{!}\colon K (D) \to K (X).
\]
In that situation Minasian and Moore and Witten
\cite{MinasianMoore,Witten} discovered that it is  sensible to think of the $K$-theory class
\[
j_{!} (\xi) \in K (X)
\]
as the ``charge'' of the $D$-brane $D$ with Chan-Paton bundle
$\xi.$

Let $b\colon K (\Z/2,2)\to K (\Z,3)$ be the indicated Bockstein: then
$BSpin$ is the fiber in the sequence 
\[
\xymatrix{
{BSpin} 
 \ar[r]
& 
{BSO}
\ar[rr]^-{bw_{2}} 
\ar[dr]_{w_{2}}
& &
{K (\Z,3)} \,\, .\\
&& {K (\Z/2,2)} 
 \ar[ur]_{b}
}
\]
Suppose that $\nu$ does not carry a $\Spinc$-structure, but
we have a map $H\colon X \to K (\Z,3)$ making the diagram
\[
\xymatrix{
    D \ar[d]_-{j} \ar[r]^-\nu & BSO \ar[d]^-{bw_{2}} \\
X \ar[r]_-{H} & K(\Z,3) \\
}
\]
commute up to homotopy.  A homotopy
$bw_{2}\heq Hj$
determines an isomorphism
\[
     K^{*} (D) \iso  K^{*+H} (D^{\nu})
\]
(since $\nu=-Tj$).  Using the construction of May and Sigurdsson described
in Remark \ref{rem:embedding} together with the discusson of twisted
Umkehr maps above, we have a twisted umkehr map
\begin{equation}
   j_{!}\colon K^{*} (D) \xra{} K^{*+H} (X).
\end{equation}
The class $j_{!} (\xi) \in K^{*+H} (X)$ is evidently an analogue of
the charge in this situation.  The discovery of the condition that
there exists a class $H$ on $X$ such that $H|_{D} = W_{3} (\nu)$ is
due to Freed and Witten \cite{fw:asd}.

\section{The general theory of parametrized objects}\label{sec:param}

\subsection{$\i$-Topoi}

The general theory of parametrized objects works not only over the $\i$-category of spaces, as used in the previous section, but equally well over an arbitrary $\i$-topos.
Recall that an $\i$-topos $\X$ is a presentable $\i$-category which arises as an accessible left-exact localization of a presheaf $\i$-category.
The terminal $\i$-topos is the $\i$-category of spaces.

The key feature of an $\i$-topos $\X$ is that $\X$ satisfies descent.
A very succinct way of expressing this fact is as follows: adopting the notation of \cite{HTT}, we write $\O_X=\Fun(\Delta^1,\X)$ for the $\i$-category of arrows of $\X$ and $p:\O_\X\to\X$ for the Cartesian fibration which assigns to an object $f:S\to T$ in $\O_\X$ its target $T$ in $\X$.
Clearly the fiber of this Cartesian fibration over the object $T$ is precisely $\X_{/T}$, the slice over $T$, which is itself an $\i$-topos.
Straightening this Cartesian fibration, we obtain a functor
\[
\X^{\op}\too\widehat{\Cat}_\infty
\]
which is a {\em sheaf} in the sense that it preserves limits: that is, if $T\simeq\colim T_\alpha$ is a colimit diagram in $\X$, then the induced map
\[
\X_{/T}\too\lim_\alpha\X_{/T_\alpha}
\]
is an equivalence in $\widehat{\Cat}_\infty$.
In fact, this descent condition characterizes $\i$-topoi amongst locally Cartesian closed presentable $\i$-categories.

Let $\C$ be an arbitrary presentable $\i$-category, which we will view as the $\i$-category in which our sheaves on $\X$ take values.
Following \cite{HTT}, we define the $\i$-category of $\C$-valued sheaves on $\X$ as
\[
\Shv_\C(\X)=\Fun^{\lim}(\X^{\op},\C)
\]
the $\i$-category of limit preserving functors from $\X^{\op}$ to $\C$.
In light of the discussion above, the target fibration $p:\O_\X\to\X$ can be viewed (via the straightening functor) as a $\widehat{\Cat}_\i$-valued sheaf on $\X$.
In the special case in which $\C=\spaces$ is the $\i$-category of spaces, we simply write $\Shv(\X)$ in place of $\Shv_{\spaces}(\X)$, and the Yoneda embedding induces an equivalence
\[
\X\simeq\Fun^{\lim}(\X^{\op},\spaces)\simeq\Shv(\X)
\]
from $\X$ to sheaves of spaces on $\X$.

We are now in a position to define objects of an $\i$-category $\C$ parametrized over objects of an $\i$-topos $\X$.

\begin{definition}
Let $\X$ be an $\i$-topos and let $\C$ be a presentable $\i$-category.
Then a family of objects of $\C$ parametrized by an object $S$ of $\X$ is a $\C$-valued sheaf on $\X_{/S}$.
The $\i$-category of objects of $\C$ parametrized by an object $S$ of $\X$ is the $\i$-category $\Shv_\C(\X_{/S})$ of $\C$-valued sheaves on $\X_{/S}$.
\end{definition}

\subsection{Parametrized objects in the $\i$-category of spaces}

To justify this notion, it is instructive to first consider the basic
case of the terminal $\i$-topos $\X=\spaces$, the $\i$-category of
spaces.

Given an $\i$-category $\aC$ and a Kan complex $S$ (which we think of
as corresponding to a space via the singular complex), we wish to
define an $\i$-category $\aC_{/S}$ of objects of $\aC$ parametrized
over $S$.
%as $\i$-categories of objects of $\aC$ parametrized by $X$ and pointed objects of $\aC$ parametrized by $X$, respectively.
Of course, we should not think of $S$ as being fixed; rather, we
require restriction (a.k.a. pullback or base-change) functors
\[
f^*\colon \aC_{/T}\longrightarrow\aC_{/S}
\]
for each map of Kan complexes $f\colon S\to T$.

Moreover, these must be compatible with composition.
Given a $2$-simplex which exhibits $g\circ f\colon S\to U$ as a
composite of $f \colon S\to T$ followed by $g\colon T\to U$, we require a natural
$2$-simplex exhibiting $(g\circ f)^* \colon \aC_{/U}\to\aC_{/S}$ as a
composite of $g^*:\aC_{/U}\to\aC_{/T}$ followed by
$f^* \colon \aC_{/T}\to\aC_{/S}$, and so on for all higher dimensional
simplices. 
In other words, we require a functor
\[
\spaces^{\op}\longrightarrow\widehat{\Cat}_\i
\]
from the $\i$-category $\spaces^{\op}$ of spaces to the $\i$-category $\widehat{\Cat}_\i$ of (not necessarily small) $\i$-categories.

Lastly, this functor must satisfy descent, which is to say that the parametrization construction $\C_{/(-)}$ must be local in the base.
For instance, given maps $f \colon S\to T$ and $g \colon S\to U$, the
restriction functors $f^* \colon \aC_{/T}\to\aC_{/S}$ and
$g^* \colon \aC_{/U}\to\aC_{/S}$, fit into a square
\[
\xymatrix{
\aC_{/T\underset{S}{\coprod} U}\ar[r]\ar[d] & \aC_{/U}\ar[d]\\
\aC_{/T}\ar[r] &\aC_{/S}}
\]
which should be cartesian.  Specifically, there is a natural map from the
$\i$-category of objects parametrized over the pushout $T\coprod_S U$
to the pullback of the $\i$-categories of parametrized objects, and we
require that this map is an equivalence.
This can be regarded as a higher categorical analogue of the Mayer-Vietoris axiom for cohomology theories, which states that the presheaf of spectra $R^{(-)}:\spaces^{\op}\to\Sp$ associated to a spectrum $R$, via the function spectrum, sends (homotopy) pushouts squares in $\spaces$ to (homotopy) pullbacks in $\Sp$.
Similarly, the coproduct axiom dictates that $R^{(-)}:\spaces^{\op}\to\Sp$ sends (possibly infinite) coproducts in $\spaces$ to products in $\Sp$, and this is also a descent condition:
Given a family of spaces $S_\lambda$ indexed by some (possibly infinite) set $\Lambda$ with coproduct $S=\coprod_\lambda S_\lambda$, the inclusions $i_\lambda:S_\lambda\to S$ induce a map
\[
\aC_{/S}\longrightarrow\prod_\lambda\aC_{/S_\lambda}
\]
which we also require is an equivalence.
Together, these two conditions amount to saying that the $\i$-category $\aC_{/S}$ of objects of $\aC$ parametrized over $S$ is {\em local} over the base $S$.
Finally, we have the obvious normalization condition: if $S=*$ is the terminal space, then we must have an equivalence $\aC_{/*}\simeq\aC$.

The following proposition is a concise reformulation of the discussion above.

\begin{proposition}\label{prop:sheaf}
If $\X=\spaces$ is the terminal $\i$-topos, then evaluation at the point
\[
\Pre_{\widehat{\Cat}_\i}(\spaces)\too\widehat{\Cat}_\i
\]
induces an equivalence between $\widehat{\Cat}_\i$ and the full subcategory
\[
\Shv_{\widehat{\Cat}_\i}(\spaces)\subset\Pre_{\widehat{\Cat}_\i}(\spaces)
\]
of the (very large) $\i$-category $\Pre_{\widehat{\Cat}_\i}(
\spaces)$ of $\widehat{\Cat}_\i$-valued presheaves on $\spaces$ spanned by the sheaves (that is, the limit-preserving functors).
\end{proposition}

\begin{proof}
The $\i$-category $\spaces^{\op}$ is freely generated under limits by the
initial object $*$ and the $\i$-category $\widehat{\Cat}_\i$ admits
all small limits. 
\end{proof}

This leads naturally to the following definition, a direct generalization of the notion of parametrized space (respectively, pointed space, spectrum) given in \cite{ABGHR1}.

\begin{definition}
Let $\aC$ be a (possibly large) $\i$-category and let $S$ be a Kan complex.
The $\i$-category $\aC_{/S}$ of objects of $\aC$ parametrized over $S$ is the $\i$-category $\Fun(S^{\op},\aC)$ of $\aC$-valued presheaves on $S$.
\end{definition}

\begin{remark}
If $S$ is connected, then $S\simeq BG$ where $G\simeq\Omega S$ is the loop space of $S$ (at a chosen basepoint).
In this case, a functor $BG^{\op}\to\aC$ is an object of $\aC$ equipped with a right action of the $\i$-group (grouplike monoidal $\i$-groupoid) $G$.
Note that if $\aC=\spaces$ is the $\i$-category of spaces, then our notion of a space parametrized over $BG$ is a functor $BG^{\op}\to\spaces$, or a (naive) $G$-space.
Of course, since $BG$ is a space, we also have the slice $\i$-category $\spaces_{/BG}$ of spaces over $BG$.
These are canonically equivalent by the straightening construction of \cite[2.2.1.2]{HTT}.
\end{remark}

\begin{proposition}\label{prop:1ff}
Let $\aC$ be a (possibly large) $\i$-category.
Then there exists a unique sheaf of (possibly large) $\i$-categories
\[
\aC_{/(-)} \colon \spaces^{\op}\longrightarrow\widehat{\Cat}_\i
\]
on $\spaces$ whose value at $S\in\spaces$ is equivalent to the $\i$-category $\aC_{/S}$ of $\aC$-valued presheaves on $S$.
\end{proposition}

\begin{proof}
By proposition \ref{prop:sheaf}, to specify a limit-preserving functor
$F:\spaces^{\op}\to\Cat_\i$, it is enough to specify the image of the
initial object $*$, which we take to be $\aC$.  For a given space $S$,
we have that $\Fun(S^{\op},\aC)\simeq\lim_S\aC$, which shows that the
value of $F$ on $S$ is equivalent to $\aC_{/S}$.
%Lastly, let $p:\aC[-]\to\spaces^{\op}$ be the unstraightening of this functor.
\end{proof}

\subsection{Parametrized objects over presheaf $\i$-topoi}
The $\i$-topos $\Pre(\aT)$ of presheaves of spaces on a small $\i$-category $\aT$ has the effect of formally adding all small colimits to $\aT$.
Therefore, if $T$ is an object of $\Pre(\aT)$, then writing $\aT_{/T}$ for the pullback
\[
\xymatrix{
\aT_{/T}\ar[r]\ar[d] & \aT\ar[d]\\
\Pre(\aT)_{/T}\ar[r] & \Pre(\aT)},
\]
the resulting fully faithful functor $\aT_{/T}\to\Pre(\aT)_{/T}$ induces a colimit preserving functor
\[
\Pre(\aT_{/T})\too\Pre(\aT)_{/T}.
\]
By construction, this functor is fully faithful, since it is the pullback of the fully faithful Yoneda embedding $\aT\to\Pre(\aT)$ along the projection $\Pre(\aT)_{/T}\to\Pre(\aT)$.
It is also essentially surjective: given a map $S\to T$ in $\Pre(\aT)$, writing $S\simeq\colim U_\alpha$ as a colimit of representable presheaves $U_\alpha$, we obtain a diagram in $\aT_{/T}$ whose colimit in $\Pre(\aT_{/T})$ is sent to its colimit $S$ in $\Pre(\aT)_{/T}$.
Hence it is an equivalence of $\i$-categories.

The equivalence $\Pre(\aT)_{/T}\simeq\Pre(\aT_{/T})$ can be formulated more conceptually.
Recall that an object $U$ of an $\i$-category $\C$ is said to be {\em
completely compact} if the associated corepresentable functor
$\map_\C(U,-)\colon \C\to\spectra$ commutes with small colimits~\cite[5.1.6.5]{HTT}.

\begin{proposition}
Let $\X$ be an $\i$-topos and let $\aT\subset\X$ denote the full subcategory of completely compact objects of $\X$.
Then $\X$ is a presheaf $\i$-topos if and only if the colimit preserving functor $\Pre(\aT)\to X$ induced by the inclusion $\aT\subset\X$ is an equivalence.
In particular, any presheaf $\i$-topos is freely generated under colimits by its full subcategory of completely compact objects.
\end{proposition}

\begin{proof}
Clearly $\X$ is a presheaf $\i$-topos if $\Pre(\aT)\to\X$ is an equivalence.
Conversely, if $\X$ is a presheaf $\i$-topos, then $\X\simeq\Pre(\aT')$ for some small $\i$-category $\aT'\subseteq\X$.
Since the objects of $\aT'$ are completely compact, we see that $\aT'\subset\aT$, so $\Pre(\aT')\subset\Pre(\aT)$ is fully faithful.
It therefore suffices to show that $f\colon \Pre(\aT)\to\X$ is fully faithful.
To this end, choose $Y\in\Pre(\aT)$, and write $Y\simeq\colim U_\alpha$ for $U_\alpha\in\aT$.
Then, for any $U\in\aT$, $\map_{\Pre(\aT)}(U,Y)\simeq\map_{\Pre(\aT)}(U,\colim_\alpha U_\alpha)\simeq\colim_\alpha\map_\aT(U,U_\alpha)\simeq\colim_\X\map(U,U_\alpha)\simeq\colim_\X\map(U,f(Y))$ since $f$ preserves colimits and $U$ is complete compact.
\end{proof}

If $\X=\spaces_{/T}$ is the slice $\i$-topos of spaces over $T$, then
an object $S\to T$ of $\X$ is completely compact if and only if $S$ is
contractible~\cite[5.1.6.9]{HTT}.  The following corollary is an
immediate consequence.

\begin{corollary}
Let $\X$ be a presheaf $\i$-topos.
Then, for any (possibly large) $\i$-category $\C$, there is a canonical equivalence
\[
\Shv_\C(\X)\simeq\Pre_\C(\aT).
\]
In particular, if $\X=\spaces_{/T}$, then $\Shv_\C(\X)\simeq\Pre_\C(T)$.
\end{corollary}

\subsection{The base-change functors: $f^*$ and its adjoints $f_!$ and $f_*$}
\label{sec:one-two-three}

In practice, it is useful to require more structure on $\aC$ than that of an arbitrary $\i$-category.
The first and most useful assumption is that $\aC$ is
presentable, which is to say that $\aC$ has all small colimits and $\aC\simeq\Ind_\kappa(\aC^\kappa)$ for some infinite regular cardinal $\kappa$ (here $\aC^\kappa$ denotes the full subcategory of $\kappa$-compact objects in $\aC$).  Since maps between presentable $\i$-categories are typically taken to be colimit preserving, they admit right adjoints by definition.  We want our theory of parametrized objects to reflect this; that is, when $\aC$ is presentable, each of the base-change functors $f^*\colon \aC_{/T}\to\aC_{/S}$ should admit a right adjoint $f_*\colon \aC_{/S}\to\aC_{/T}$.

\begin{proposition}
Let $\X$ be an $\i$-topos and let $\C$ be a presentable $\i$-category.
Then there is a canonical equivalence of presentable $\i$-categories
\[
\C\otimes\X\simeq\Shv_\C(\X).
\]
\end{proposition}

\begin{proof}
This is a special case of \cite[Proposition 6.3.1.16]{HA}, which is more generally true whenever $\X$ is presentable, though we will not need this extra generality.
\end{proof}

\begin{proposition}\label{prop:2ff}
Let $\X$ be an $\i$-topos and let $\aC$ be a presentable $\i$-category.  Then there exists a unique sheaf of presentable $\i$-categories and right-adjoint functors  
\[
\aC_{/(-)} \colon \X^{\op}\longrightarrow\PrL
\]
on $\X$ whose value at the object $S\in\X$ is equivalent to the presentable $\i$-category
\[
\aC_{/S}\simeq\aC\otimes\X_{/S}\simeq\Shv_\C(\X)
\]
of $\aC$-valued sheaves on $\X_{/S}$. 
\end{proposition}

\begin{proof}
Since the restriction functors $f^*:\X_{/S}\to\X_{/T}$ are left adjoint functors of presentable $\i$-categories, it is clear that, tensoring with $\aC$, we obtain a functor (defined up to contractible ambiguity) $\X^{\op}\to\PrL$.
Thus it only remains to see that this functor preserves limits.
Equivalently, since $\aC\otimes(-):\PrL\to\PrL$ commutes with colimits, we can check instead that the left adjoints $f_!$ of the restrictions $f^*$ induce colimit decompositions
\[
\colim_\alpha \X_{/U_\alpha}\to\X_{/T}
\]
in $\PrL$ for any colimit diagram $\colim_\alpha U_\alpha\simeq T$ in $\X$.
Since the forgetful functor $\PrL\to\widehat{\Cat}_\i$ preserves limits, this follows immediately from descent, i.e. that the restrictions induce an equivalence $\X_{/T}\simeq\lim_\alpha\X_{/U_\alpha}$ in $\widehat{\Cat}_\i$.
\end{proof}

The restriction functors $f^*:\X_{/T}\to\X_{/S}$ admit both left and right adjoints $f_!$ and $f^*$, respectively.
Let $\PrLR$ denote the subcategory of $\PrL$ whose objects are again presentable $\i$-categories, but whose morphisms consist of those functors $\aC\to\aD$ which are simultaneously both left and right adjoints (equivalently, by the adjoint functor theorem, those functors $\aC\to\aD$ which preserve all limits and colimits).

\begin{lemma}\label{lem:prlr}
Each of the arrows in the cartesian square
\[
\xymatrix{
\PrLR\ar[r]\ar[d] & \PrL\ar[d]\\
\PrR\ar[r] & \widehat{\Cat}_\i}
\]
preserve and detect small limits.

\end{lemma}

\begin{proof}
Both arrows to $\widehat{\Cat}_\i$ preserve limits by \cite[Theorems 5.5.3.13 and 5.5.3.18]{HTT}, and their proofs reveal that they also detect limits.
%Alternatively, since limits (in $\widehat{\Cat}_\i$) of presentable $\i$-categories are presentable, it suffices to consider the case of a diagram $K^{\triangleleft}\to\PrL$ (respectively, $K^{\triangleleft}\to\PrR$) such that the forgetful functor to $\widehat{\Cat}_\i$ is a limit diagram.
%A straightforward argument then shows that limits of left (respectively, right) adjoint functors are themselves left (respectively, right) adjoints.
\end{proof}

Lemma \ref{lem:prlr} and the proof of proposition \ref{prop:1ff} now
imply the following theorem.

\begin{theorem}\label{thm:exist1}
Let $\X$ be an $\i$-topos and $\C$ a presentable $\i$-category.
Then the sheaf $\aC_{/(-)}$ of presentable $\i$-categories on $\X$ factors through the subcategory $\PrLR\subset\PrL$.
In particular, there exists a unique sheaf of presentable $\i$-categories and left and right adjoint functors
\[
\C_{/(-)}:\X^{\op}\longrightarrow\PrLR
\]
on $\X$ whose value at the object $S\in\X$ is equivalent to the $\i$-category $\Shv_\C(\X_{/S})$ of $\aC$-valued sheaves on $\X_{/S}$.
\end{theorem}

Thus far, we have constructed, for each $\i$-topos $\X$ and each presentable $\i$-category
$\aC$, a ``three functor formalism'' for the theory
of objects of $\aC$ parametrized over objects of $\X$.  There are a number of ``Beck-Chevalley'' type
relations which occur when given a pullback square in $\X$;
see \cite[Propositions 2.2.11, 11.4.8]{MS} for a treatment in the
context of (pointed) spaces and spectra parametrized over spaces.

%Under certain condition (both on $\aC$ and on $f$), $f_*$ also preserves colimits and therefore itself admits a right adjoint $f^!$.

\begin{proposition}[Beck-Chevalley conditions]
Suppose given a cartesian square
\[
\xymatrix{
S\ar[r]^f\ar[d]^g & T\ar[d]^h\\
U\ar[r]_i & V}
\]
in an $\i$-topos $\X$.
Then there are canonical natural equivalences $g_! f^*\simeq i^*h_!$ and $i^*
h_*\simeq g_* f^*$ of functors $\X_{/T}\to\X_{/U}$ that are
interchanged by adjunction.
\end{proposition}

\begin{proof}
Using the commutativity of the square and (co)unit transformations, it is easy to construct natural transformations $g_! f^*\to i^* h_!$ and $i^* h_*\to g_* f^*$.
Moreover, by adjunction and symmetry, the second transformation is an equivalence if and only if the first is an equivalence.
Thus it only remains to show that the transformation $g_! f^*\to i^* h_!$ induces an equivalences upon evaluation at an object $X$ of $\X_{/T}$.
But the projection $\X_{/T}\to\X$ is conservative, meaning it is enough to check this in $\X$ itself, where it follows from the equivalences
$X\times_T S\simeq X\times_T T\times_V U\simeq X\times_V U$.
\end{proof}

\subsection{The proper pushforward and its right adjoint}

We now suppose that $\aC$ is a stable presentable $\infty$-category.
In this case, it turns out that for
proper geometric morphisms $p_* \colon \X\to\Y$ of $\i$-topoi, the induced functor
$p_* \colon \Shv_\C(\X)\to\Shv_\C(\Y)$ preserves {\em all} colimits, so that $p_*$ itself admits a right adjoint $p^! \colon \Shv_\C(\Y)\to\Shv_\C(\X)$. 

\begin{proposition}
Let $\aC$ be a compactly generated stable $\i$-category and $p_* \colon \X\to\Y$ a proper geometric morphism of $\i$-topoi.  Then the induced functor
$p_* \colon \Shv_\C(\X)\to\Shv_\C(\Y)$ admits a right adjoint
$p^! \colon \Shv_\C(\Y)\to\Shv_\C(\X)$.
\end{proposition}

\begin{proof}
By \cite[Remark 7.3.1.5]{HTT}, $p_*\colon \X\to\Y$ preserves filtered colimits, so that $p_* \colon \spectra\otimes\X\to\spectra\otimes\Y$ is a map of $\spectra$-modules in $\PrL$ and in particular preserves all colimits.
Now, since $\C$ is stable and presentable, $\C$ is also a $\spectra$-module in $\PrL$, so that 
\[
p_*:\C\otimes\X\simeq\C\otimes_{\spectra}{\spectra}\otimes\X\to\C\otimes_{\spectra}{\spectra}\otimes\Y\simeq\C\otimes\Y
\]
is again a map of $\spectra$-modules in $\PrL$ and in particular preserves all colimits.
It follows that $p_*:\Shv_\C(\X)\simeq\C\otimes\X\to\C\otimes\Y\simeq\Shv_\C(\Y)$ admits a right adjoint $p^!:\Shv_\C(\Y)\to\Shv_\C(\X)$.
\end{proof}

\begin{definition}
Let $\X$ be an $\i$-topos.  Then a map $p\colon S\to T$ in $\X$ is
proper if $p_* \colon \X_{/S}\to\X_{/T}$ is a proper morphism of $\i$-topoi.
\end{definition}

For example, in the $\i$-topos of spaces, a map is proper if the
fibers are compact.  Given a proper map $p\colon S\to T$ in $\X$, it
follows that the pushforward $p_*$ admits a right adjoint $p^!$.  This
gives a series of adjunctions $p_!,p^*,p_*,p^!$, each of which is
right adjoint to the functor on its left.

\subsection{The tangent bundle of an $\i$-topos}

In this subsection, we provide an interpretation of parametrized
spaces and spectra over an arbitrary $\i$-topos in terms of the notion
of tangent bundles of $\i$-categories.

Let $\X$ be an $\i$-topos, and let
\[
\O_\X\simeq\Fun(\Delta^1,\X)
\]
denote the source of the presentable fibration $p \colon \O_\X\to\X$ over
$\X$ associated to ``target'' map $d_0 \colon \Delta^1\to\Delta^0$. 
Note that $p$ is a presentable cartesian fibration since $\X$ admits
pullbacks and each of the fibers $\X_{/X}$ over $X\in\X$ is
presentable (even an $\i$-topos), and that the fiber of the
projection $p\colon \O_{\X}\to\X$ is precisely the $\i$-category $\aS_{/X}$
of spaces over $X$.

The objects are $\aS_{/X}$ not literally spaces fibered over $X$, for
the simple reason that $X$ itself need not be a space, but we may
nevertheless reasonably regard them as spaces over $X$ for the
following reason.  If $\X=\aS$, then a space over $X\in\X$ is precisely
an $X$-indexed family (a functor $X^{\op}\to\X$) of objects of $\X$, and
this literally still holds if $X$ is a ``space'' in $\X$; furthermore,
even if $X$ is an arbitrary object of $\X$, not necessarily a space,
we may nevertheless identify $\X_{/X}$ with the $\i$-category of
bundles of spaces over $\X_{/X}$. 

\begin{definition}
A bundle of spaces over an $\i$-category $\X$ with limits is a right
fibration $\aF\to\X$ such that (the straightening $F\colon \X^{\op}\to\aS$
of) $\aF$ preserves limits. 
\end{definition}

\begin{lemma}
Let $\X$ be an $\i$-topos and let $t_*:\X\to\aS$ denote the unique
geometric morphism from $\X$ to the $\i$-category $\aS$ of spaces. 
\begin{enumerate}
\item
If $X\in\X$ is a space, in the sense that $X\simeq t^*S$ for some
$S\in\aS$, then $\X_{/X}\simeq\Fun(S^{\op},\X)$. 
\item
For any object $X\in\X$, $\X_{/X}\simeq\Fun^{\lim}(\X_{/X}^{\op},\aS)$
is the $\i$-topos of bundles of spaces over $\X_{/X}$. 
\end{enumerate}
\end{lemma}

The target fibration $p:\O_\X\to\X$ is an unstable version of the
tangent bundle $q:\mathrm{T}_\X\to\X$ of $\X$.  Since $p$ is a presentable
fibration, we may regard it as the unstraightening of the (limit
preserving) functor $\aS_{/\X}:\X^{\op}\to\PrR$.  Stabilizing, we
arrive at a functor $\Sp_{\X}:\X^{\op}\to\PrR_{\mathrm{st}}$, which
unstraightens to the tangent bundle $q:\mathrm{T}_\X\to\X$.  We record this in
the following proposition.

\begin{proposition}
Let $\X$ be an $\i$-topos.  Then the fiber of the projection
\[
p:\O_\X\too\X
\]
over $X\in\X$ is the $\i$-category $\aS_{/X}$ of spaces over $X$, and
the fiber of the projection 
\[
q:\mathrm{T}_\X\too\X
\]
over $X\in\X$ is the $\i$-category $\Sp_{/X}$ of spectra over $X$.
\end{proposition}

\section{Algebraic structures on parametrized objects}

In practice, it in not enough to consider objects of presentable
$\i$-category $\C$ parametrized over objects of an $\i$-topos $\X$.
Often we are interested in whether or not these parametrized objects
admit some sort of algebraic structure.  The algebraic structures in
question are encoded by the action of an $\i$-operad
$\O^\otimes$.

\subsection{The closed monoidal structure}
\label{sec:four-five-functor}

We fix an $\i$-operad $\aO^\otimes$, an $\i$-topos $\X$, and an $\aO$-monoidal presentable $\i$-category $\C^{\otimes}$.  The goal now is to
construct, for each object $S$ of $\X$, an $\aO$-monoidal $\i$-category
$\aC_{/S}^\otimes$ with underlying $\i$-category $\aC_{/S}$ such that the
restriction functor $f^*\colon \aC_{/T}\to\aC_{/S}$ induced by a map of Kan
complexes $f:S\to T$ is  $\aO$-monoidal.

Let $\mathrm{Alg}_\aO(\PrLR)$ denote the pullback
\[
\xymatrix{
\mathrm{Alg}_\aO(\PrLR)\ar[r]\ar[d] & \mathrm{Alg}_\aO(\PrL)\ar[d]\\
\PrLR\ar[r] & \PrL,}
\]
that is, the subcategory of $\mathrm{Alg}_\aO(\PrL)$ consisting of those $\aO$-monoidal functors which are also right adjoints.
To analyze the behavior of limits in $\mathrm{Alg}_\aO(\PrLR)$ we require the following technical lemma.

\begin{lemma}\label{lem:lim}
The subcategory $\widehat{\Cat}^\mathrm{R}_\i\subset\widehat{\Cat}_\i$ spanned by the complete $\i$-categories and the limit preserving functors is stable under pullbacks.
\end{lemma}

\begin{proof}
Suppose given a pullback diagram
\[
\xymatrix{
\aA\ar[r]^f\ar[d]^g & \aB\ar[d]^h\\
\aC\ar[r]^i & \aD}
\]
in $\widehat{\Cat}_\i$ such that $\aB,\aC,\aD$ are complete $\i$-categories and $h,i$ are limit preserving functors.
We first show that $\aA$ is complete, which amounts to showing that the constant diagram functor $\aA\to\Fun(K,\aA)$ admits a right adjoint $\lim:\Fun(K,\aA)\to\aA$.
Since the corresponding result holds for $\aB,\aC,\aD$ by assumption, we obtain a map
\[
\Fun(K,\aA)\simeq\Fun(K,\aB)\times_{\Fun(K,\aD)}\Fun(K,\aC)\to\aB\times_\aD\aC\simeq\aA
\]
which is easily seen to be right adjoint to $\aA\to\Fun(K,\aA)$ since
mapping spaces in a limit of $\i$-categories are computed as the limit
of the mapping spaces.  To see that the projections $f\colon \aA\to\aB$ and
$g\colon \aA\to\aC$ preserve limits, we note that, by construction,
$\lim\colon \Fun(K,\aA)\to\aA$ is the pullback of the diagram of maps
$\lim\colon \Fun(K,\aB)\to\aB$ and $\lim:\Fun(K,\aC)\to\aC$ over
$\lim\colon \Fun(K,\aD)\to\aD$, so this is clear.
Finally, given a commutative diagram
\[
\xymatrix{
\aA'\ar[r]^f\ar[d]^g & \aB\ar[d]^h\\
\aC\ar[r]^i & \aD}
\]
in $\widehat{\Cat}^\mathrm{R}_\i$, we must show that the $\i$-groupoid of limit preserving functors from $\aA'$ to $\aA$ over $\aB\to\aD\leftarrow\aC$ is contractible.
This follows from the fact that these functors form a full $\i$-subgroupoid of the $\i$-groupoid of all functors from $\aA'$ to $\aA$ over $\aB\to\aD\leftarrow\aC$ coupled with the observations that this latter $\i$-groupoid is contractible and that the unique such functor preserves limits.
\end{proof}

\begin{remark}
A similar argument shows that the subcategory $\widehat{\Cat}^\mathrm{R}_\i\subset\widehat{\Cat}_\i$ is stable under all small limits.
\end{remark}

\begin{corollary}\label{cor:lim}
The $\i$-category $\Alg_\aO(\PrLR)$ admits all small limits
and the inclusion of the subcategory $\Alg_\aO(\PrLR)\subset\Alg_\aO(\PrL)$ preserves them. 
\end{corollary}

\begin{proof}
This is immediate from lemma \ref{lem:lim} above, as $\Alg_\aO(\PrLR)$ is the pullback of a diagram of complete $\i$-categories and limit preserving functors.
\end{proof}

Our main foundational theorem is the following result, which follows
from corollary~\ref{cor:lim} and the proof of
proposition~\ref{prop:1ff}.

\begin{theorem}\label{thm:paramobj}\label{thm:exist2}
There exists a unique sheaf of presentable $\aO$-monoidal
$\i$-categories $\aC_{/(-)}^\otimes$ on $\X$ whose value at the object $S$
of $\X$ is the $\aO$-monoidal $\i$-category $\aC_{/S}^\otimes$ of
$\aC^\otimes$-valued sheaves on $S$. 
\end{theorem}

\begin{remark}
In the symmetric monoidal context, the right adjoint
$f_*$ is lax $\aO$-monoidal and that the left adjoint $f_!$ is oplax
$\aO$-monoidal.
\end{remark}

We now further suppose that $\aO^\otimes$ comes equipped with a fixed
map $\bE_1^\otimes\to\aO^\otimes$.  This implies (by restriction along this map) that any $\aO$-monoidal $\i$-category $\aC^\otimes$ is equipped with a distinguished monoidal structure $\otimes$. The following lemma is immediate from the adjoint functor theorem.

\begin{lemma}
Let $\aC^\otimes$ be a monoidal presentable $\i$-category.
Then $\aC^\otimes$ is closed in the sense that, for each object $X$ of
$\aC$, the left and right multiplication functors
$X\otimes(-)\colon \aC\to\aC$ and $(-)\otimes X\colon \aC\to\aC$ admit right adjoints.
\end{lemma}

Writing $\otimes$ for the monoidal product obtained by restriction
along the map $\bE_1^\otimes\to\aO^\otimes$, then, for each object $X\in\aC$, we write
\[
F(X,-)\colon \aC\to\aC
\]
for the right adjoint of the right multiplication functor $(-)\otimes
X\colon \aC\to\aC$.
As $S$ varies over all spaces,
the base-change functors and closed tensor structures collectively
give rise to a (not necessarily symmetric) sort of ``Wirthm\"uller
context'' \cite{FauskHuMay}. 

In the context of parametrized spaces, $f^*$ is a symmetric monoidal
functor.  The situation of a symmetric monoidal functor with
left and right adjoints gives rise to a series of compatibility
formulas (e.g., the projection formula).  Following \cite{FauskHuMay},
we now abstract this relationship into what we will refer to as a
Wirthm\"uller context.

We continue to fix an $\i$-operad $\aO^\otimes$ over $\bE_1$ and an
$\aO$-monoidal $\i$-category $\aC^\otimes$.  To say more, we now
suppose that the map $\bE_1^{\otimes}\to\aO^\otimes$ factors
through $\bE_\i^\otimes$, which is to say that $\aC^\otimes$ is a {\em
symmetric} monoidal presentable $\i$-category \cite[Proposition
4.1.1.20]{HTT}.
Specializing the definition of $\mathrm{Alg}_\aO(\PrLR)$ to the case of the terminal $\i$-operad,
we obtain the $\i$-category $\mathrm{CAlg}(\PrLR)$ of symmetric monoidal presentable $\i$-categories and symmetric monoidal functors which are simultaneously left and right adjoints.
%denote the pullback
%\[
%\xymatrix{
%\mathrm{CAlg}(\PrLR)\ar[r]\ar[d] & \mathrm{CAlg}(\PrL)\ar[d]\\
%\PrLR\ar[r] & \PrL,}
%\]
%i.e., the subcategory of $\mathrm{CAlg}(\PrL)$ consisting of those
%symmetric monoidal functors which are also right adjoints.

\begin{definition}\label{defn:wirth}
A Wirthm\"uller context is a $\mathrm{CAlg}(\PrLR)$-valued
sheaf on $\X$; that is, a limit preserving functor 
\[
\X^{\op}\longrightarrow\mathrm{CAlg}(\PrLR).
\]
%The $\i$-category of Wirthm\"uller contexts is the full subcategory of the $\i$-category $\Fun(\X^{\op},\mathrm{CAlg}(\PrLR))$ spanned by the sheaves.
\end{definition}

Some of the useful consequences of the existence of a Wirthm\"uller
context are summarized in the following standard proposition. 
%{\bf We need to add the directions of the maps and explain again the
%coherences that are present here.}

\begin{proposition}\label{prop:omniwirth}
Let $\X$ be an $\i$-topos, $\aC^\otimes$ a symmetric monoidal presentable $\i$-category,
$f\colon S\to T$ a morphism in $\X$, $X$ an
object of $\aC_{/S}^{\otimes}$, and $Y$ and $Z$ objects of
$\aC_{/T}^\otimes$.  Then there are natural equivalences:  
\begin{enumerate}
\item $f^* (Y \otimes_T Z) \htp f^*Y \otimes_S f^*Z$
\item $F_T(Y, f_* X) \htp f_* F_S(f^* Y,X)$,
\item $f^* F_T(Y,Z) \htp F_S (f^*Y, f^*Z)$,
\item $f_! (f^* Y \otimes_S X) \htp Y \otimes_T f_! X$,
\item $F_T (f_! X,Y) \htp f_* F_S(X,f^* Y)$.
\end{enumerate}
\end{proposition}

\begin{proof}
As explained in~\cite{FauskHuMay}, we can deduce all of these
equivalences from (1) and the projection formula (4).  First, the
equation (1) follows immediately from the fact that
$f^*\colon \C^\otimes_{/T}\to\C^\otimes_{/S}$ is a strong symmetric
monoidal functor.  Next, (4) is immediate whenever $\C=\Pre(\T)$ is a
presheaf $\i$-topos since (replacing $\X$ with
$\C\otimes\X\simeq\Fun(\T^{\op},\X)$) the projection formula holds
inside any $\i$-topos: $Y\times_T S\times_S X\simeq Y\times_T X$.  To
conclude (4) in general, we use the fact that any symmetric monoidal
presentable $\i$-category $\C$ is a symmetric monoidal accessible
localization of some $\Pre(\T)$; since the localization functor is a
strong symmetric monoidal left adjoint~\cite[2.2.1.9]{HA}, we can
deduce (4) from the result for presheaves.

We now explain how to obtain the remaining equivalences.  Let
$f^*\colon \B^\otimes\to\A^\otimes$ be a morphism of commutative
algebra objects in $\PrLR$.  By the relative adjoint functor
theorem~\cite[8.3.2.7]{HA}, $f_*\colon \A^\otimes\to\B^\otimes$ is lax
symmetric monoidal.  Using~\cite[1.7]{BarwickGlasmanNardin} to study
the opposite category, we can analogously deduce that
$f_!\colon \A^\otimes\to\B^\otimes$ is oplax symmetric monoidal.  Now
suppose we are given an object $X$ of $\A$ and objects $Y$ and $Z$ of
$\B$.  Then (2) follows from (1) because
\begin{align*}
&\map(Z,F(Y,f_*X))\simeq\map(Z\otimes Y,f_*X)\simeq\map(f^*Z\otimes
f^*Y,X)\simeq\\ &\map(f^*Z,F(f^*Y,X))\simeq\map(Z,f_*F(f^*Y,X)),
\end{align*}
and (3) follows from (2) and (4) since the unit applied to $Z$ gives a
map
\[
F(Y,Z))\to F(Y,f_*f^*Z)\simeq f_*F(f^*Y,f^*Z)
\]
whose adjoint is the map $f^*F(Y,Z)\to F(f^*Y,f^*Z)$, which is an
equivalence because
\begin{align*}
&\map(X,f^*F(Y,Z))\simeq\map(f_!X,F(Y,Z))\simeq\map(f_!X\otimes
Y,Z)\simeq\\ &\map(f_!(X\otimes f^*Y),Z)\simeq\map(X\otimes
f^*Y,f^*Z)\simeq\map(X,F(f^*Y,f^*Z)).
\end{align*}
Finally, (5) follows because
\begin{align*}
&\map(Z,F(f_!X,Y))\simeq\map(Z\otimes
f_!X,Y)\simeq\map(f_!(f^*Z\otimes X),Y)\simeq\\ &\map(f^*Z\otimes
X,f^*Y)\simeq\map(f^*Z,F(X,f^*Y))\simeq\map(Z,f_*F(X,f^*Y)).
\end{align*}
(See also \cite[2.2.2,11.4.1]{MS} for a verification in the particular
case of parametrized spaces and spectra.)
\end{proof}

We can also consider this situation when $\aC^{\otimes}$ is an
$\bE_n$-monoidal presentable $\i$-category.  In this case, provided
that $n > 2$, the theory is the same because the equivalences of
proposition~\ref{prop:omniwirth} arise as isomorphisms in the homotopy
category, and for $n > 2$ an $\bE_n$-monoidal presentable
$\i$-category has a closed symmetric monoidal homotopy category.  (We
also use the fact that $\bE_n$-monoidal functors induce symmetric
monoidal functors on the homotopy category in this case.)  We suspect
that analogous formulas hold for the braided monoidal case $n = 2$ and
even the monoidal case $n=1$.

Finally, when $\X=\aS$, we have the following basic existence result
as a corollary of theorem~\ref{thm:paramobj}. 

\begin{corollary}\label{cor:wirthexists}
A Wirthm\"uller context over $\aS$ determines, and is determined by, a symmetric
monoidal presentable $\i$-category. 
\end{corollary}

\begin{remark}
This generalizes in a obvious way to prehsheaf $\i$-topoi, and in
straightforward but less obvious way to $\i$-topoi.  We leave the
details to the interested reader. 
\end{remark}

\subsection{Parametrized objects over spaces with multiplicative structure}

In this section, we study the multiplicative structures that arise on
$\i$-categories of parametrized objects over $\bE_n$ spaces.  In the
previous sections, the multiplicative structure on $\aC_{/S}$ was
obtained pointwise, or, equivalently, from the evident external
product via pullback along the diagonal $\Delta \colon S \to S \times
S$ of the base space.  Here, the multiplicative structures arise from
an actual product on $S$ itself.

Recall that the $\infty$-categorical Day convolution product \cite[\S
6.3]{HA} is a consequence of the existence of a symmetric monoidal
functor from spaces to presentable $\infty$-categories.  The relevant
functor on objects agrees with the sheaf
$\aS_{/(-)} \colon \aS\to\PrL$ on objects, but takes maps of Kan
complexes $f \colon S \to T$ to the left adjoint
$f_! \colon \aS_{/S}\to\aS_{/T}$ of $f^*$.
Since we will be interested in $\i$-categories of modules over an
$\aO$-monoidal presentable $\i$-category $\aR$, for the remainder of
this section we will replace $\PrL$ with the equivalent $\i$-category
$\Mod_\aS$ of $\aS$-module objects in $\PrL$. 

\begin{proposition}
There is a unique colimit-preserving functor
\[
\Pre \colon \aS\to\Mod_\aS
\]
whose value at the space $S$ is the $\i$-topos $\aS_{/S}$ of
presheaves of spaces on $S$. 
\end{proposition}

\begin{proof}
The $\i$-category of spaces $\aS$ is freely generated under colimits by
the one-point space \cite[5.1.5.8]{HTT}.  Since any space $S$ is
equivalent to the $S$-indexed colimit of the constant diagram on the
point, it follows that  
\[
\Pre(S) \htp S\otimes \aS.
\]
\end{proof}

We have the following proposition as a consequence of the properties of the $\i$-categorical Day
convolution project \cite[6.3.1.2]{HA}.

\begin{proposition}\label{prop:paramsym}
The functor $\Pre \colon \aS\to\Mod_{\spaces}$ extends to a symmetric
monoidal functor 
\[
\Pre \colon \spaces^\otimes\to\Mod_{\spaces}^\otimes.
\] 
\end{proposition}

It follows from proposition~\ref{prop:paramsym} that
the functor $\Pre$ preserves multiplicative structures: 

\begin{corollary}
Let $\cO$ be an $\i$-operad and let $X$ be an $\cO$-algebra object of
$\aS$.  Then $\Pre(X)$ is an $\cO$-algebra object of $\Mod_\aS$.
\end{corollary}

We now assume that $\cO$ is a unital and coherent $\i$-operad.  Since
$\Mod_{\spaces}$ is a symmetric monoidal $\i$-category, we can consider
$\cO$-algebra objects in $\Mod_{\spaces}$.  We require $\cO$ to be unital
since we will need to consider the unit map $\eta \colon \spaces\to\R$ of
an $\cO$-algebra object $\R$ of $\Mod_{\spaces}$.

Recall that if $\R$ is an object of $\Alg_{\cO}(\Mod_{\spaces})$ then
$\Mod^\cO_\R=\Mod^\cO_\R(\Mod_{\spaces}\times_\Gamma\cO)$ is the $\infty$-category of
$\R$-module objects in $\Mod_{\spaces}$ \cite[\S 3.3.3]{HA}.

\begin{proposition}
Let $\cO$ be a coherent and unital $\i$-operad and let $\R$ be an
$\cO$-algebra object of $\Mod_{\spaces}$.  Then there exists a unique
colimit-preserving functor 
\[
\Pre_\R \colon \spaces\to\Mod^\cO_\R
\]
whose value on the point is the ``free rank one''
$\R$-module $\R$.
\end{proposition}

\begin{proof}
The functor 
\[
\eta^* \colon (-)\otimes_{\spaces}\R \colon \Mod^\cO_{\spaces}\to\Mod^\cO_\R 
\]
preserves colimits, as it is left adjoint to the restriction
\[
\eta_* \colon \Mod^\cO_\R\to\Mod^\cO_{\spaces}
\]
along $\eta \colon \spaces\to\R$.  Thus, the composite
\[
\Pre_\R\simeq\eta^*\circ \Pre \colon \spaces\to\Mod^\cO_{\spaces}\to\Mod^\cO_\R
\]
preserves colimits, and sends the one-point space $*$ to
$\R\simeq\spaces\otimes_{\spaces}\R$.
\end{proof}

The functor $\Pre_\R$ also preserves multiplicative structures:

\begin{corollary}
Let $\cO$ be an $\i$-operad and let $X$ be an $\cO$-algebra object of
$\spaces$.  Then $\Pre_\R(X)$ is an $\cO$-algebra object of $\Mod^\cO_\R$.
\end{corollary}

The main example of this phenomenon which will be of interest to us is
the case of an $\cO$-algebra object $\R$ of $\Mod_{\spectra}$, where $\cO$ is
a coherent $\i$-operad under $\bE_1$.  Then $\R$ is in particular an
associative algebra object of $\Mod_{\spectra}$, and so it has a Picard
$\i$-groupoid $\Pic(\R)$, the full subgroupoid of $\R$ spanned by the
invertible objects.  

\section{Categorification of Picard group}
\label{sec:picard-i-groupoids}

In this section we define and study the Picard $\i$-groupoid of an
$\aO$-monoidal stable presentable $\i$-category $\R$ (for suitable
$\i$-operads $\aO^\otimes$) and the categories of parametrized objects
over Picard $\i$-groupoids.  Roughly speaking, we define the Picard
$\i$-groupoid of $\R$ as the space of invertible objects in $\R$; the
work of the section is to keep track of the multiplicative structure
inherited from $\R$.  The main theorem of this section describes the
categorified $\Pic$ as participating in an adjunction that (when
specialized to modules over an $\bE_n$ ring spectrum) gives rise to
the Thom spectrum functor as the counit.  

Note that contrary to the standard convention, our Picard
$\i$-groupoids will be grouplike $\aO$-spaces, not necessarily
grouplike $\bE_\i$-spaces.  As a consequence, we begin by recalling
some details concerning grouplike $\bE_1$-spaces.  In any $\i$-topos
(in particular, such as the $\i$-category of spaces), there is a
notion of a grouplike $\bE_1$-space \cite[5.1.3.2]{HA}.  Specifically,
we have the following characterization \cite[5.1.3.5]{HA}.

\begin{definition}
An $\bE_1$-space $X$ is said to be {\em grouplike} if the monoid $\pi_0
X$ is a group.  Given a map $\eta\colon \bE_1 \to\cO$ of coherent
$\i$-operads, we say that an $\cO$-monoidal space $X$ is {\em
grouplike} if $\eta^*X$ is a grouplike $\bE_1$-space.
\end{definition}

Given any $\cO$-monoidal space $X$, we can restrict to the maximal
grouplike subspace of $X$.

\begin{lemma}\label{lem:maxgrouplike}
For an $\cO$-monoidal space $X$, there is a maximal grouplike subspace $\GL X$.
That is, the inclusion 
\[
\Mon_{\aO}^{\gp}(\spaces) \to \Mon_{\aO}(\spaces)
\]
of grouplike $\cO$-monoidal spaces into $\cO$-monoidal spaces has a right adjoint $\GL$ given by passage to the maximal grouplike $\cO$-monoidal space.
\end{lemma}

\begin{proof}
The inclusion functor preserves colimits \cite[5.1.3.5]{HA} and therefore the adjoint functor theorem implies that there exists a right adjoint $\GL$.  We can explicitly identify this as follows:  Given an $\cO$-monoidal space $X$, $\pi_0(X)$ is a monoid.
The maximal grouplike space $\GL X$ is the full subgroupoid obtained by passage to
the invertible elements of $\pi_0(X)$ (i.e., the maximal group contained in $\pi_0(X)$).  Since any product of invertible objects in $\pi_0(X)$ is invertible, the criterion
of \cite[2.2.1.1]{HA} implies that this space is itself $\cO$-monoidal.  Because $\GL X$ is a full subgroupoid of $X$, it is clear that any map from a grouplike $\cO$-monoidal space uniquely factors through it.
\end{proof}

More generally, given any $\aO$-monoidal $\i$-category $\R$, we can pass to the full subcategory of invertible objects in $\R$, which we will denote by $\R^{\times}$.  
Explicitly, this can be built as the
pullback
\begin{equation}\label{eqn:pic}
\xymatrix{
(\R^{\otimes})^{\times} \ar[r] \ar[d] & \R^{\otimes} \ar[d] \\
\Ho(\R^{\otimes})^\times \ar[r] & \Ho(R^{\otimes}),
}
\end{equation}
where $\Ho(\R^{\otimes})^\times$ denotes the full monoidal
subcategory of the (ordinary) monoidal category
$\Ho(\R^{\otimes})$ spanned by the invertible objects.  This is the 
$\aO$-monoidal $\i$-category of invertible objects.  The same argument
as in the proof of lemma~\ref{lem:maxgrouplike} proves the following lemma.

\begin{lemma}
For an $\aO$-monoidal $\i$-category $\R$, the full $\i$-subcategory
$\R^{\times}$ of invertible objects is an $\aO$-monoidal
$\i$-category.
\end{lemma}

However, we want the Picard object to be a space.  Recall that the
inclusion of $\i$-groupoids into $\infty$-categories preserves
products and has a right adjoint; explicitly, if $\C$ is an
$\i$-category, then $\C^{\htp}$ is the subcategory of $\C$ consisting
of the invertible morphisms.

We can now define the Picard $\i$-groupoid of an $\bE_1$ object in
$\Pr^\L$.

\begin{definition}\label{defn:pic}
Let $\R$ be an $\spectra$-algebra in $\Pr^\L$.  Then $\Pic(\R)$ is the
maximal grouplike $\i$-groupoid $(\R^{\times})^{\htp}$ inside of the monoidal
$\i$-category $\R^{\times}$.
When $\R=\Mod_R$ for an $\bE_n$-ring spectrum $R$, $n>1$, we typically
write $\Pic_R$ in place of $\Pic(\R)$.
\end{definition}

When applied to the category of modules over a commutative ring
spectrum $R$, definition~\ref{defn:pic} recovers the usual
construction of the Picard group.  In fact, we can perform this
construction in either order.  First, given $\R$, pass to the full
$\i$-subcategory $\R^{\times}$ of invertible objects in $\R$, and then
take the maximal $\i$-groupoid in $\R^{\times}$.  Equivalently, given
$\R$, pass to the maximal $\i$-groupoid $\R^{\htp}$ contained in $\R$,
then pass to the largest grouplike object inside $\R^{\htp}$.

Furthermore, if $\R$ is a closed symmetric monoidal stable
$\i$-category, we can characterize $\Pic(\R)$ as a subspace of the
subcategory of dualizable objects in $\R$.  (See for example \cite[\S
2]{maypic} for an excellent discussion of this perspective on the
level of homotopy categories.)
In this case, the inverse of $X \in \Pic(\R)$ is the functional dual
$F_{\R}(X,1)$. 
The point is that the equivalences witnessing the invertibility of $X$ are duality data;
this follows from \cite[2.9]{maypic} since $\i$-categorical duality
can be detected on the homotopy category.
It is not difficult to extend the description of
equation~\ref{eqn:pic} and  the inverse to the situation
when $\R$ has weaker monoidal structures, but to state the results 
requires a discussion of duality in these settings which we do not
wish to pursue herein.

In order to obtain a multiplicative structure, we would like to
describe $\Pic(\aR)$ more explicitly as part of an adjunction.  To make
this precise, we first need the following result which allows us to
control the size of the Picard group.

\begin{lemma}
Let $\A$ be a monoidal presentable $\i$-category.  Then there exists a
regular cardinal $\kappa$ such that the inclusion
\[
\Pic(\A^\kappa)\subseteq\Pic(\A)
\]
is an equivalence of $\infty$-groupoids.
In particular, $\Pic(\A)$ is essentially small.
\end{lemma}

\begin{proof}
By \cite[Lemma 6.3.7.12]{HA}, there exists a regular cardinal $\kappa$ such that $\A$ is $\kappa$-presentable, the unit $1_\A$ of $\A$ is $\kappa$-compact, and the full subcategory $\A^\kappa\subset\A$ consisting of the $\kappa$-compact objects is a monoidal subcategory.
Let $A\in\Pic(\A)$ be an invertible object of $\A$.  Since
$\A\simeq\Ind_\kappa(\A^\kappa)$, $A=\colim_I A_i$ for some
$\kappa$-filtered diagram of $\kappa$-compact objects of $\A$.  Since
$A$ has an inverse $B$, $1\simeq A\otimes B\simeq\colim_I (A_i\otimes
B)$, and since $1$ is $\kappa$-compact, the equivalence $1\to\colim_I
(A_i\otimes B)$ factors through a $\kappa$-small stage $J\subset
I$. But then $1\simeq\colim_J (A_j\otimes B)$ implies that 
\[
\colim_J A_j\simeq B^{-1}\simeq\colim_I A_i,
\]
so that $A$ is a $\kappa$-small colimit of $\kappa$-compact objects
and hence itself is $\kappa$-compact. 
\end{proof} 

This now permits us to give the following characterization of $\Pic$.

\begin{theorem}\label{thm:picthm}
Let $\aO$ be a coherent $\infty$-operad equipped with a map
$\bE_1 \to\aO$. 
Then
\[
\Pic \colon\Alg_{\aO}(\PrL)\longrightarrow\Alg^{\gp}_{\aO}(\spaces)
\]
is right adjoint to the free presentable $\i$-category functor
\[
\Pre\colon \Alg^{\gp}_{\aO}(\spaces)\longrightarrow\Alg_{\aO}(\PrL).
\]
\end{theorem}

\begin{proof}
Let $G$ be a grouplike $\aO$-monoidal space and $\R$ a presentable
$\aO$-monoidal $\infty$-category.  Restriction along the Yoneda
embedding $G\to \Pre(G)$ gives maps 
\[
\map_{\Alg_{\aO}(\PrL)}(\Pic(G),\R)\longrightarrow\map_{\widehat{\Cat}_\i}(\Pic(G),U(\R))\longrightarrow\map_{\Alg^{\gp}_{\aO}(\spaces)}(G,\Pic(\R)),
\]
where $U \colon \PrL\longrightarrow\widehat{\Cat}_\i$ denotes the
underlying monoidal $\i$-category functor, such that the composite
is an equivalence.
\end{proof}

The unit of the adjunction is the Yoneda embedding $G\to\R[G]$.  The
counit of the adjunction is the map 
\[
\spaces_{/\Pic(\R)} \longrightarrow\R
\]
adjoint to the identity map $\Pic(\R)\to\Pic(\R)$.
As a functor between presentable $\infty$-categories, this map
preserves colimits and is uniquely determined by the image of
$\Pic(\R)$ in $\R$.  

\begin{remark}
When $\aO$ is a model for the $\bE_n$ operad, we have the following
specialization:
The functor 
\[
\Pic \colon\Alg_{\bE_n}(\PrL)\to\Alg^{\gp}_{\bE_n}(\spaces)
\]
is corepresented by the $\bE_n$-monoidal $\i$-category
$\spaces[\Omega^n\Sigma^n_+*]$.  This is because $\Omega^n\Sigma^n_+*$ is
the free grouplike $\bE_n$-space on a single generator $*$. 
\end{remark}

Passing to the stable setting, the argument for
Theorem~\ref{thm:picthm} yields the following proposition.

\begin{theorem}\label{prop:main}
Let $\bE_1 \to\cO$ be a map of coherent $\i$-operads and let $\R$
be a stable presentable $\cO$-monoidal $\i$-category. 
Then the canonical map
\[
\spectra_{/\Pic(\R)}\longrightarrow\R
\]
is a map of stable presentable $\cO$-monoidal $\i$-categories.
\end{theorem}

We close with a remark about the way in which the work of this section
is a categorification of the classical theory of the space of units of
a ring spectrum.  The multiplicative structure on $\Pic(\R)$ is such
that the canonical map   
\[
\spectra_{/\Pic(\R)}\longrightarrow\R,
\]
adjoint to the inclusion $\Pic(\R)\to\R$ of the invertible objects,
is an $\cO$-algebra map.  Conceptually, this is a categorification of
the adjunction which defines $\GL$. Just as the underlying
infinite loop space functor
$
\Omega^\i \colon \spectra\to\spaces
$
is right adjoint to the symmetric monoidal suspension spectrum functor
$
\Sigma^\i_+ \colon \spaces\to\spectra$,
the forgetful functor
\[
\map_{\spectra} \colon (\spectra,-) \colon {\Pr}^\L_\mathrm{St}\to\spaces
\]
is right adjoint to the symmetric monoidal functor 
\[
\Pre_{\spectra} \colon \spaces\to {\Pr}^\L_\mathrm{St}.
\]

\section{Multiplicative properties of the Thom spectrum
functor}\label{sec:thom}

In this section, we apply the work of the previous section in the
context of the generalized Thom spectrum functor.
Theorem~\ref{prop:main} has the following immediate consequence, which
proves Theorem~\ref{thm:infmain}; this is a generalization of Lewis'
theorem about multiplicative structures on Thom spectra.

\begin{corollary}
Let $\bE_1 \to\cO$ be a map of $\i$-operads and let $\R$ be a
stable presentable $\cO$-monoidal $\i$-category.  The the composite
functor
\[
\spaces_{/\Pic(\R)}\longrightarrow\spectra_{/\Pic(\R)}\longrightarrow\R
\]
is a map of presentable $\cO$-monoidal $\i$-categories.
Moreover, if $\R=\Mod_R$ for an $\bE_n$-algebra object $R$ of
$\spectra$, $n>1$, then this composite is equivalent to the generalized
Thom spectrum functor. 
\end{corollary}

As an application, we use this to proof $R$-module generalizations of
Lewis' results about the multiplicative properties of the Thom
isomorphism theorem.  Lewis proved \cite[\S
IX.7.4]{lewis-may-steinberger} that given an $\bE_n$ classifying map
$f \colon X \to B\GL{\bS}$ such that $Mf$ admits an $\bE_n$
orientation over $R$ (i.e., an $\bE_n$ map $Mf \to R$), then the map
inducing the Thom isomorphism is an $\bE_n$ map.  We now provide a
concise proof of the analogous results for generalized Thom spectra
over $B\GL{R}$.

Assume that $R$ is an $\bE_{n+1}$ ring spectrum and $f$ is an object
of the $\i$-category $\Alg_{/ \bE_n}(\spaces_{/ B\GL{R}})$, i.e., an
$\bE_n$ map of spaces
\[
f \colon X \to B\GL{R}.
\]
One of the main theorems of our previous work on Thom spectra and
units \cite{ABGHR, ABGHR1, ABGHR2} shows that an orientation of the
Thom spectrum $Mf$ is specified by a map $P \to \GL{R}$ in
$\Mod_{\GL{R}}$, where here $P$ is the pullback of the universal
principal $\GL{R}$-bundle along $f$ and $\Mod_{\GL{R}}$ is the
$\i$-category of $\GL{R}$-modules in spaces.  This suggests the
following generalization of an orientation to the setting of $\bE_n$
maps.

\begin{definition}
Assume that $R$ is an $\bE_{n+1}$-ring spectrum.
Let $P$ be an object in $\Alg_{/\bE_n}(\Mod_{\GL{R}})$.
Then the space of $\bE_n$ orientations of $P$ is the space of
$\bE_n$-algebra maps $P\to\GL{R}$ in $\Alg_{/\bE_n}(\Mod_{\GL{R}})$. 
\end{definition}

It is convenient to view the Thom spectrum functor in this light; the
following lemma is an immediate consequence of the
straightening/unstraightening equivalence~\cite[Theorem
2.2.1.2]{HTT}. 

\begin{lemma}
There is an equivalence of $\bE_n$-monoidal
$\i$-categories
\[
\spaces_{/B\GL{R}} \htp \Mod_{\GL{R}},
\]
and hence an equivalence of $\i$-categories
\[
\Alg_{/ \bE_n}(\spaces_{/B\GL{R}}) \htp \Alg_{/\bE_n}(\Mod_{\GL{R}}).
\]
\end{lemma}

As a consequence, the Thom spectrum functor can be written as the
composite:
\[
\Alg_{/ \bE_n}(\spaces_{/B\GL{R}}) \to \Alg_{/ \bE_n}(\ispec_{/B\GL{R}}) \to \Alg_{/
\bE_n}(\ispec),
\]
where the first map is the stabilization.

Now given any object $P$ in $\Alg_{/ \bE_n}(\Mod_{\GL{R}})$, we have the 
following version of the Thom diagonal, given by the $\bE_n$ map
\[
\xymatrix{
\Delta \colon P \ar[r]^-{\id \times *} & P \times ({\GL{R}}\times X)
}
\]
where here ${\GL{R}}\times X$ is the free $\GL{R}$-module on the
space $X$.  We use the fact that both $X$ and $\GL{R}$ are based
spaces. 

Applying the Thom spectrum functor now yields a map
\[
Mf \to Mf \sma (R \sma X_+),
\]
of $\bE_n$-ring spectra.

On the other hand, given an orientation $P \to \GL{R}$, applying the
Thom spectrum functor produces a map
\[
Mf \to R
\]
of $\bE_n$-ring spectra.  Putting these together, we get the composite
\[
Mf \to Mf \sma (R \sma \Sigma^\i_+ X) \to R \sma (R \sma \Sigma^\i_+ X) \to R \sma \Sigma^\i_+ X
\]
which is a map of $\bE_n$-ring spectra realizing the Thom isomorphism: 

\begin{theorem}
An $\bE_n$ orientation $P \to \GL{R}$ in $\Alg_{\bE_n}(\Mod_{\GL{R}})$
gives rise to a map of $\bE_n$-ring spectra
\[
Mf \to R \sma \Sigma^\i_+ X
\]
which is an equivalence and realizes the Thom isomorphism.
\end{theorem}

\appendix

\section{The Brauer group and twisted parametrized spectra}\label{sec:brauer}

In this section, we indicate how to categorify the definition of
the Picard group; this produces a delooping which should be regarded
as the Brauer group of an $E_\i$ ring spectrum.  The connections to
the classical definitions of the Brauer group have been studied by the
third author with various collaborators \cite{G1, G2}.  We do not go
into detail about any of the applications of this here, other than to
briefly observe that this definition allows us to situate the work of
Douglas on ``twisted parametrized spectra'' \cite{douglas} in our
context.

\begin{definition}\label{defn:brauer}
Let $R$ be a $\bE_\i$-ring spectrum.  The Brauer group of $R$ 
is 
\[
\mathrm{Br}_R = \Pic(\Mod_{\Mod_R}(\PrL)^\mathrm{cg}),
\]
the Picard $\infty$-groupoid of the symmetric monoidal
$\infty$-category $\Mod_R^\omega$ of compactly generated $\Mod_R$-modules in
$\PR^\L$, the $\i$-category of presentable $\i$-categories.  
\end{definition}

It is straightforward to check that the Brauer group of $R$ provides a
delooping of the Picard group.  

\begin{lemma}
Let $R$ be an $\bE_\i$-ring spectrum.  There is a natural
equivalence
\[
\Pic_R\simeq\Omega\mathrm{Br}_R.
\]
\end{lemma}

The Brauer group now provides the proper context to define twisted
parametrized spectra:

\begin{definition}[Haunts and specters]\label{defn:spooks}
For a commutative $\bS$-algebra $R$, the $\infty$-category of
$R$-haunts over a space $X$ is given by the $\i$-category (actually,
$\i$-groupoid) $(\mathrm{Br}_R)_{/X} = (\Pic_{\Mod_R})_{/X}$ of
$\Mod_R$-torsors over $X$.  For a given haunt $\aH$ on a space $X$,
the $\i$-category of specters is the limit of the composite 
\[
X\to\mathrm{Br}_R\to\Mod_{\Mod_R}.
\]
\end{definition}

It is now possible to reprove the theorems of Douglas using
definition~\ref{defn:spooks} and the work of this paper, although we
do not carry out this project here.

\iffalse

In particular, we can expand
Douglas' sketch proof of the following illuminating comparison result: 

\begin{theorem}
Let $R$ be an $\bE_\i$-ring spectrum and let $\aH$ be a haunt
classified by a pointed map $f \colon X \to \mathrm{Br}_R$ for a pointed
connected space $X$.  Then the $\i$-category of specters associated to
$\aH$ is equivalent to the $\i$-category of modules over the Thom
spectrum associated to 
\[
\Omega f \colon \Omega X \to \Omega \mathrm{Br}_R \htp\Pic_R.
\]
\end{theorem}

\fi

\section{Comparison to the May-Sigurdsson model}\label{sec:comparison}

In this section, we show that our theories of parametrized spaces and
spectra are compatible with those of May-Sigurdsson.  As a
consequence, one can produce a parametrized spectrum in our context
from point-set data (e.g., sequences of parametrized spaces linked by
fiberwise suspension), and a functor from a point-set functor which is
homotopical (e.g., a Quillen functor).  Also as a consequence,
homotopical conclusions about May and Sigurdsson's setup follow from
our results.

To make the comparison, one produces from a model category  
$\aC$ to an associated $\i$-category.  When $\aC$ is a simplicial
model category, one way to do this is to restrict to the full subcategory of
cofibrant-fibrant objects $\aC^\mathrm{cf}$; then the
simplicial nerve \cite[1.1.5.5]{HTT} $\N(\aC^\mathrm{cf})$
is an $\i$-category.  Although any combinatorial model category is
Quillen equivalent to a simplicial model category \cite{dugger}, this
replacement process can be inconvenient.  Furthermore, very few
functors preserve cofibrant-fibrant objects; this is a particular
problem when studying (symmetric) monoidal model categories.

More recently, \cite[\S 1.3.3]{HA} provides an analogue of the
Dwyer-Kan simplicial localization.  Starting with a (not necessarily
simplicial) model category $\aC$, one passes to an $\i$-category via
the ordinary nerve applied to the full subcategory of cofibrant
objects and subsequently inverts the weak equivalences:
\[
\N(\aC^\mathrm{c})[W^{-1}].
\]
Given a simplicial model category $\aC$, there is an equivalence of
$\i$-categories  
\[
\N(\aC^\mathrm{cf}) \htp \N(\aC^\mathrm{c})[W^{-1}],
\]
which implies that we can apply either process as
needed \cite[1.3.3.7]{HA}. 

Recall that $\spaces_{/S}$ can be described via the straightening and
unstraightening correspondence as the $\i$-category associated to the
model category $\Set_{\Delta/ S}$, and $\spaces_{*/S}$ can analogously be
described as the $\i$-category of pointed objects in $\Set_{\Delta/
S}$.  This provides a comparison to the $\i$-categories associated to
the May-Sigurdsson categories of parametrized spaces $\top_{/ B}$ and
$(\top_{/ B})_*$ over a space $B$.

\begin{proposition}
Let $B$ be a space.
There are equivalences of symmetric monoidal $\i$-categories  
\[
\spaces_{/\Sing B}^{\otimes} \htp \N(\Set_{\Delta/ \Sing B})[W^{-1}]^{\otimes} \htp \N(\top_{/B})[W^{-1}]^{\otimes}
\]
and
\[
(\spaces_{/\Sing B})_*^{\otimes} \htp \N((\Set_{\Delta/ \Sing B})_*)[W^{-1}]^{\otimes} \htp \N((\top_{/B})_*)[W^{-1}]^{\otimes}.
\]
\end{proposition}

\begin{proof}
For a space $B$, the projective model structure on $\top_{/B}$ (in
which fibrations and weak equivalences are detected by the forgetful
functor to $\top$ with the standard model structure) is Quillen
equivalent to the corresponding simplicial model category structure on
simplicial sets over $\Sing B$, which in turn is Quillen equivalent to
the simplicial model category of simplicial presheaves on the
simplicial category $\mathfrak{C}[\Sing B]$ (with the projective model
structure) \cite[2.2.1.2]{HTT}.  (Here $\mathfrak{C}$ denotes the left
adjoint to the simplicial nerve; it associates a simplicial category
to a simplicial set \cite[1.1.5]{HTT}.)

Next, we have a comparison
\[
\mathrm{St}\colon\mathrm{N}\Set_{\Delta/\Sing
  B}^\circ\longrightarrow\Fun(\Sing
B^{\op},\mathrm{N}\Set_\Delta^\circ); 
\]
the map, called the {\em straightening} functor, rigidifies a
fibration over $\Sing B$ into a presheaf of $\i$-groupoids on $\Sing
B$ whose value at the point $b$ is equivalent to the fiber over $b$
\cite[3.2.1]{HTT}.

Finally, the symmetric monoidal structure on $\spaces_{/S}$ is Cartesian and
therefore unique \cite[2.4.1.9]{HA}.  Thus, we can promote this
equivalence to an equivalence of symmetric monoidal $\i$-categories.
The result for pointed objects follows.
\end{proof}

To complete the comparison to the model of May-Sigurdsson, we need to
study the base-change functors.  Almost all of the subtlety and
difficulty of the foundational portion of their work arises from the
complexities of topological spaces (which they must contend with in
order to handle the equivariant setting) and the fact that it is
impossible to have a model structure in which the pairs $(f_!, f^*)$
and $(f^*, f_*)$ are simultaneously Quillen adjunctions.

Although the point-set category $(\top_{/B})_*$ of ex-spaces has a
model structure induced by the standard model structure on $\top$
(which they refer to as the $q$-model structure), one of the key
insights of May and Sigurdsson is that for the purposes of stable
parametrized homotopy theory it is essential to work with the (Quillen
equivalent) $qf$-model structure \cite[6.2.6]{MS}.

The situation is easier in the simplicial setting: For a map $f \colon
A \to B$, we can obtain point-set models of the functors $f^*$, $f_*$,
and $f_!$ by considering model categories of simplicial presheaves.
We must still confront the fact that
\[
f^*\colon \Fun(\mathfrak{C}[\Sing
  B^{\op}],\Set_\Delta)\longrightarrow\Fun(\mathfrak{C}[\Sing
  A^{\op}],\Set_\Delta) 
\]
is a {\em right} Quillen functor for the {\em projective} model
structure, with left adjoint $f_!$, and a {\em left} Quillen functor
for the {\em injective} model structure, with right adjoint $f_*$, on
the above categories of simplicial presheaves.  Nonetheless, this
suffices to produce the desired adjoint pairs on the level of
$\i$-categories.

\begin{theorem}
The Wirthm\"uller context we construct in
corollary~\ref{cor:wirthexists} on $(\spaces_{/S})_*$ is compatible  with
that of May-Sigurdsson.
\end{theorem}

\begin{proof}
To see this, observe that it suffices to check this for $f^*$;
compatibility then follows formally for the adjoints $f_*$ and $f_!$.
Thus, we need to check that the right derived functor of $f^* \colon
(\top_{/B})_* \to (\top_{/A})_*$ in the $qf$-model structure
is compatible with the right derived functor of
\[
f^* \colon \Fun(\mathfrak{C}[\Sing B^{\op}], \Set_\Delta) \to
\Fun(\mathfrak{C}[\Sing A^{\op}], \Set_\Delta)
\]
in the projective model structure.  By the work of
\cite[9.3]{MS}, it suffices to check the compatibility for
$f^*$ in the $q$-model structure.  Since both versions of $f^*$ that
arise here are Quillen right adjoints, this amounts to the
verification that the diagram 
\[
\xymatrix{
\Fun(\mathfrak{C}[\Sing B^{\op}], \Set_\Delta) \ar[d]^{\Un} \ar[r]^-{f^*} &
\Fun(\mathfrak{C}[\Sing A^{\op}], \Set_\Delta) \ar[d]^{\Un} \\ 
\Set_{\Delta/B} \ar[r]^-{f^*} & \Set_{\Delta/A}  \\
}
\]
commutes when applied to fibrant objects, where here $\Un$ denotes the
unstraightening functor (which is the right adjoint of the Quillen
equivalence).  Finally, this follows from \cite[2.2.1.1]{HTT}.

The promotion of this comparison to the symmetric monoidal structure
is a consequence of the fact that $f^*$ preserves products and the
fact that the Cartesian symmetric monoidal structure is unique.
\end{proof}

Therefore, in order to compare our model of parametrized spectra over
$B$ to the May-Sigurdsson model, we will work with the corresponding
formal stabilization of model categories.  Specifically, given a left
proper cellular model category $\aC$ and an endofunctor of $\aC$,
Hovey constructs a cellular model category $\Sp^{\bN}
\aC$ of spectra \cite{Hovey}.  When the $\aC$ is
additionally a simplicial symmetric monoidal model category, the
endofunctor given by the tensor with $S^{1}$ yields a simplicial
symmetric monoidal model category of symmetric spectra $\Sp^{\Sigma}
\aC$ (in addition to the simplicial model category $\Sp^{\bN}
\aC$ of spectra).  These models of the stabilization are
functorial in left Quillen functors which are suitably compatible with
the respective endofunctors (see \cite[5.2]{Hovey}).

\begin{proposition}\label{prop:nstabcomm}
Let $\aC$ be a left proper cellular simplicial model category
and write $\Sp^{\bN} \aC$ for the cellular simplicial model
category of spectra generated by the tensor with $S^1$.  Then there is
an equivalence of $\i$-categories
\[
\N((\Sp^{\bN} \aC)^\mathrm{c})[W^{-1}] \htp \Stab(\N(\aC^\mathrm{c})[W^{-1}]).
\]
When $\aC$ is a simplicial symmetric monoidal model category,
this equivalence extends to an equivalence 
\[
\N((\Sp^{\Sigma} \aC)^\mathrm{c})[W^{-1}]^{\otimes} \htp \Stab(\N(\aC^\mathrm{c})[W^{-1}]^{\otimes})
\]
of symmetric monoidal $\i$-categories.
\end{proposition}

\begin{proof}
The functors $\mathrm{Ev}_n\colon\Sp^{\bN}\aC \to \aC$ which associate
to a spectrum its $n^\mathrm{th}$-space $A_n$, induce a functor
\begin{align*}
f\colon \N((\Sp^{\bN}\aC)^\mathrm{c})[W^{-1}] \to \lim\{\cdots\overset{\Omega}{\longrightarrow}&\N(\aC^\mathrm{c}_*)[W^{-1}]\overset{\Omega}{\longrightarrow} 
\N(\aC^\mathrm{c}_*)[W^{-1}]\} \\
&\htp \Stab(\N(\aC^\mathrm{c})[W^{-1}])
\end{align*}
which is evidently essentially surjective.
To see that it is fully faithful, it suffices to check that for
cofibrant-fibrant spectrum objects $A$ and $B$ in $\Sp^{\bN}\aC$, there is an equivalence of mapping spaces
\[
\map(A,B)\simeq\holim\{\cdots\overset{\Omega}{\longrightarrow}\map(A_1,B_1)\overset{\Omega}{\longrightarrow}\map(A_0,B_0)\},
\]
where $\Omega\colon\map(A_{n+1},B_{n+1})\to\map(A_n,B_n)$ acts as
\[
A_{n+1}\to B_{n+1} \mapsto A_n\simeq\Omega A_{n+1}\to\Omega
B_{n+1}\simeq B_n.
\]
Since any cofibrant $A$ is a retract of a cellular object, inductively we can reduce to the case in which $A=F_m X$, i.e., the shifted suspension spectrum on a cofibrant object $X$ of $\aC_*$.
Then $\map(A,B)\simeq\map(X,B_m)$ by adjunction.
The latter is in turn equivalent to $\map(\Sigma^{n-m} X, B_n)$, where we interpret $\Sigma^{n-m} X=\ast$ for $m>n$, in which case the homotopy limit is equivalent to that of the homotopically constant (above degree $n$) tower whose $n^\mathrm{th}$ term is $\map(\Sigma^{n-m} X, B_n)$.

In the symmetric monoidal setting, the fact that $\Stab \aC$ is the
initial stable symmetric monoidal $\i$-category which accepts a
symmetric monoidal $\i$-functor from $\aC$ coupled with the
equivalence between prespectra and symmetric spectra implies the
desired comparison.
\end{proof}

May and Sigurdsson construct a symmetric monoidal stable model
structure on the category $\mathscr{S}_{B}$ of orthogonal spectra in
$(\top_{/B})_*$ \cite[12.3.10]{MS}.  This model structure is based on
the $qf$-model structure on ex-spaces, leveraging the diagrammatic
viewpoint of \cite{MR1806878, MandellMay}.  Similarly, they construct
a stable model structure on the category $\mathscr{P}_{B}$ of
prespectra in $(\top_{/B})_*$.  The forgetful functor
$\mathscr{S}_{B} \to \mathscr{P}_{B}$ induces a Quillen equivalence
\cite[12.3.10]{MS}.

\begin{theorem}\label{thm:comp}
Let $B$ be a topological space.  There is an equivalence of symmetric monoidal
$\i$-categories between the $\i$-category associated to the model category
of orthogonal spectra and the $\i$-category of parametrized spectra:
\[
\N(\mathscr{S}_B)[W^{-1}]^{\otimes} \htp \Fun(\Sing B^{\op}, \ispec)^{\otimes}.
\]
\end{theorem}

\begin{proof}
This is essentially an immediate consequence of
proposition~\ref{prop:nstabcomm}, using the standard comparison
between orthogonal spectra and symmetric spectra \cite{MR1806878}.
\end{proof}

\end{document}